\documentclass[letterpaper,11pt]{article}

\usepackage{amsmath,amsthm}
\usepackage{amsfonts}
\usepackage{latexsym}
\usepackage{amssymb,bm}
\usepackage{mathrsfs}
\usepackage{mathtools}
\usepackage[page]{appendix}
\usepackage{float}
\usepackage{changes}

\usepackage{graphicx}
\usepackage{epstopdf}

 \usepackage{grffile}

\usepackage{changes}

\usepackage{color}
\usepackage{color}
\newtheorem{theorem}{Theorem}[section]
\newtheorem{lemma}[theorem]{Lemma}
\newtheorem{corollary}[theorem]{Corollary}
\newtheorem{proposition}[theorem]{Proposition}

\newtheorem{algorithm}{Algorithm}

\newcommand{\inner}[2]{\langle #1,#2\rangle}

\newcommand{\norm}[1]{\|{#1}\|}

\newcommand{\R}{\mathbb{R}}
\newcommand{\tos}{\rightrightarrows}

\newcommand{\comment}[1]{}

\newcommand{\comenta}[1]{}

\newcommand{\HH}{\mathcal{H}}

\newcommand{\mgap}{\vspace{.1in}}

\begin{document}

\title{On inexact relative-error hybrid proximal extragradient, 
forward-backward and Tseng's modified forward-backward methods with inertial effects
}
%\author{}
\author{
    M. Marques Alves
\thanks{
Departamento de Matem\'atica,
Universidade Federal de Santa Catarina,
Florian\'opolis, Brazil, 88040-900 ({\tt maicon.alves@ufsc.br}).
The work of this author was partially supported by CNPq grants no.
405214/2016-2 and 304692/2017-4.}
\and
Raul T. Marcavillaca
\thanks{
Departamento de Matem\'atica,
Universidade Federal de Santa Catarina,
Florian\'opolis, Brazil, 88040-900 ({\tt rtm1111@hotmail.com}).
The work of this author was partially supported by CAPES.
}
}
%\date{May 9, 2014}
%\date{}

%\begin{document}

\maketitle

\begin{abstract}
In this paper, we propose and study the asymptotic convergence and nonasymptotic global convergence rates (iteration-complexity) of 
an inertial under-relaxed version of the relative-error hybrid proximal extragradient (HPE) method
for solving monotone inclusion problems.
We analyze the proposed method under more flexible assumptions than existing ones on the extrapolation and relative-error parameters. As applications, we propose and/or study
inertial under-relaxed forward-backward and Tseng's modified forward-backward type methods for solving structured monotone inclusions.
\end{abstract}

\pagestyle{plain}

%%%%%%%%%%%%%%%%%%%%%%%%%%%%%%%%%%%%%%%%%%%%%%%%%%%%%%%%%%%%%%%%%%%%%%%%%%%%%%%%%%%%%%%%

\section*{Introduction}

Inertial proximal point-type algorithms for monotone inclusions gained a lot of attention in
research recently (see, e.g., \cite{att.cab-con.pre218, att.cab-con.pre18} and the references therein). 
%While the idea originally traces back to the  heavy-ball method
%of Polyak \cite{}, for convex optimization, 
The first method of this type -- the inertial proximal point (PP) method --
for solving generalized equations with monotone operators was proposed and studied by Alvarez and Attouch in
\cite{alv.att-iner.svva01}. The intense research activity in the subject in the last years is in part due to its connections
with fast first-order algorithms for convex programming
 (see, e.g., \cite{att.cab-con.pre218, att.cab-con.pre18,att.chb.pey-fas.mp18,att.pey-rat.sjo16,att.pey.red-fas.jde16,lor.poc-ine.jmiv15}).

Since the inertial PP method of Alvarez and Attouch has been used as the hidden engine for the design and analysis of various first-order proximal algorithms with inertial effects, including inertial versions of ADMM, forward-backward and 
Douglas-Rachford algorithms (see, e.g., \cite{att.cab-con.pre218, att.cab-con.pre18,bot.cse.hen-ine.amc15,che.cha.ma-ine.sjis15,che.ma.yan-gen.sijo15}), it is natural to attempt to design inexact versions of it.
In \cite{bot.cse-hyb.nfao15}, Bot and Csetnek proposed and studied the asymptotic convergence of an inertial version of the hybrid proximal extragradient (HPE) method of Solodov and Svaiter~\cite{mon.sva-hpe.siam10,sol.sva-hyb.svva99}. The HPE method
is an inexact PP algorithm
for which, at each iteration, the corresponding proximal subproblems are supposed to be (inexactly) solved within
a relative error criterion (this contrasts to the summable error criterion proposed by Rockafellar~\cite{roc-mon.sjco76}).

In this paper, we propose and study the asymptotic convergence and 
nonasymptotic global convergence rates (iteration-complexity) of an inertial under-relaxed HPE method
for solving monotone inclusions. The proposed method (Algorithm \ref{inertial.hpe}) differs from the existing inertial HPE-type method of Bot
and Csetnek in the sense it is based on a different mechanism of iteration. 
Moreover, we prove its convergence and iteration-complexity under more flexible assumptions than those proposed in 
\cite{bot.cse-hyb.nfao15} on the extrapolation and relative-error parameters. 
%
% show it is more flexible with certain parameters and we also performe its iteration-complexity analysis, 
%in both pointwise and ergodic sense. 
%
As applications, we study inertial (under-relaxed) versions of the
Tseng's modified forward-backward and forward-backward algorithms (see Algorithms \ref{inertial.tseng} and \ref{inertial.fb})  for solving structured monotone inclusions problems.

The main contributions of this paper will be further discussed in Section \ref{sec:pre}.

\mgap
\noindent
This paper is organized as follows. In Section \ref{sec:pre}, we present some preliminaries and basic results, review some existing algorithms and discuss in detail the main contributions of this paper. The inertial under-relaxed HPE
method (Algorithm \ref{inertial.hpe}) is presented in Section \ref{sec:alg}; the main results are Theorems \ref{th:wc} (asymptotic convergence), and \ref{th:pic} and \ref{th:erg} (iteration-complexity). Sections \ref{sec:tseng} is devoted to
present and study the inertial versions of the Tseng's modified forward-backward and forward-backward algorithms; the main results are Theorems
\ref{th:tseng.main} and \ref{th:fb.main}. We finish the paper in Section \ref{sec:cr} with some concluding remarks.

\section{Preliminaries, basic results and general notation}
 \label{sec:pre}

\subsection{Problem statement}
Let $\HH$ be a real Hilbert space and consider the  general
monotone inclusion problem (MIP) of finding $z\in \HH$ such that
\begin{align}
 \label{eq:mip}
 0\in T(z)
\end{align}
as well as the \emph{structured} MIP
\begin{align}
 \label{eq:mips}
 0\in F(z)+B(z)
\end{align}
where $T$ and $B$ are (set-valued) maximal monotone operators on $\HH$ and $F: D(F)\subset \HH \to \HH$
is a (point-to-point) monotone operator which is either \emph{Lipschitz continuous} or
\emph{cocoercive} (see Subsections \ref{subsec:tsg} and \ref{subsec:fb} for the precise statement).
Problems \eqref{eq:mip} and \eqref{eq:mips} appear in different fields of applied mathematics and optimization including convex optimization, signal
processing, PDEs, inverse problems, among others (see, e.g.,\cite{bau.com-book,glo.osh.yin-spl.spi16}).
We mention that under mild conditions on the operators $F$ and $B$, problem \eqref{eq:mips} becomes a special instance of \eqref{eq:mip} with $T:=F+B$. %This fact will be important later on in this paper.

In this paper, we propose and study the asymptotic convergence and the iteration-complexity of
inertial under-relaxed versions of the \emph{hybrid proximal extragradient} (HPE) method (Algorithm \ref{inertial.hpe}),
and \emph{Tseng's modified forward-backward} (Algorithm \ref{inertial.tseng}) and \emph{forward-backward} (Algorithm \ref{inertial.fb}) methods for solving \eqref{eq:mip}, and \eqref{eq:mips}, respectively.

\emph{The main contributions of (as well as the most related works with) this paper will be discussed along the next subsections, the main contributions being further summarized in Subsection \ref{subsec:main.cont}.}

\subsection{The Alvarez--Attouch's inertial proximal point method}
  \label{subsec:iner.intr}

The \emph{proximal point}  (PP) \emph{method} is an iterative scheme for seeking
approximate solutions of \eqref{eq:mip}. It was first proposed by Martinet~\cite{MR0298899} for solving monotone variational inequalities (with point-to-point operators) and further
studied and developed by Rockafellar in his pioneering work~\cite{roc-mon.sjco76}.
In its exact formulation, an iteration of the PP method can
be described by
\begin{align}
 \label{eq:exact.prox}
 z_k:= (\lambda_k T+I)^{-1}z_{k-1}\qquad \forall k\geq 1,
\end{align}
where $\lambda_k>0$ is a stepsize parameter and
$z_{k-1}$ is the current iterate.

The \emph{inertial}  PP \emph{method} is a modification of \eqref{eq:exact.prox} proposed and studied by
Alvarez and Attouch in \cite{alv.att-iner.svva01} as follows: for all $k\geq 1$,
\begin{align}
\left\{
       \begin{array}{ll}
			 \label{eq:iner.intr}
				 w_{k-1}:=z_{k-1}+\alpha_{k-1}(z_{k-1}-z_{k-2}),\\[4mm]
                 z_k:= (\lambda_k T+I)^{-1}w_{k-1},
        \end{array}
        \right.
\end{align}
%
%\begin{align}
%  w_{k-1}=z_{k-1}+\alpha_{k-1}(z_{k-1}-z_{k-2}),\qquad z_k= (\lambda_k T+I)^{-1}w_{k-1}\qquad \forall k\geq 1,
%\end{align}
%
where $\{\alpha_{k}\}$ is a sequence of extrapolation parameters; note that if $\alpha_k\equiv 0$, then it follows that
\eqref{eq:iner.intr} reduces to the Rockafellar's PP method \eqref{eq:exact.prox}. Inertial PP-type
methods deserve a lot of attention in nowadays research due the possibility of extending this methodology
to different practical algorithms and, in part, as we mentioned earlier,  due to its connections with fast first-order methods in convex programming.
Asymptotic (weak) convergence of $\{z_k\}$ generated in \eqref{eq:iner.intr} to a solution of \eqref{eq:mip} was first obtained in
\cite{alv.att-iner.svva01} under the assumptions that $\lambda_k\geq \underline{\lambda}>0$ and 
\begin{align}
  \label{eq:alpha.ser}
 0\leq \alpha_{k-1}\leq \alpha_k \leq \alpha<1/3\qquad \forall k\geq 1.
\end{align}
The above upper bound $1/3$ on $\{\alpha_k\}$ has become standard in the analysis of
inertial-like proximal algorithms (see, e.g., \cite{che.cha.ma-ine.sjis15,che.ma.yan-gen.sijo15,lor.poc-ine.jmiv15, mou.oli-con.jcam03}).
It seems that \eqref{eq:alpha.ser} was first improved by Alvarez in \cite[Proposition 2.5]{alv-wea.siam03}
in the setting of projective-proximal point-type methods and, more recently, by Attouch and Cabot in \cite{att.cab-con.pre18} with
relaxation playing a central
role.
One of the main goals of this contribution is the analysis of an inertial under-relaxed HPE-type 
method under the assumption (actually more general than)
\eqref{eq:alpha.ser} on $\{\alpha_k\}$; see Assumption ${\bf (A)}$.

\subsection{The hybrid proximal extragradient method of Solodov and Svaiter}
  \label{subsec:hpe.intr}

It is of course important to design and study inexact versions of
known (exact) numerical algorithms, and this also applies to \eqref{eq:exact.prox}.
In~\cite{roc-mon.sjco76}, Rockafellar proved that
if, at each iteration $k\geq 1$, $z_k$ is computed satisfying
\begin{align}
 \label{eq:inexact.prox}
 \norm{z_k-(\lambda_k T+I)^{-1}z_{k-1}}\leq e_k,\quad \sum_{k=1}^\infty\,e_k<\infty,
\end{align}
and $\{\lambda_k\}$ is bounded away from zero, then $\{z_k\}$ converges (weakly) to a solution
of \eqref{eq:mip}.
Many modern inexact versions of the PP method \eqref{eq:exact.prox}, as opposed to the summable error criterion \eqref{eq:inexact.prox}, use \emph{relative error tolerances} for solving
the associated subproblems.
The first methods of this type were proposed by Solodov and Svaiter in \cite{sol.sva-hyb.svva99,sol.sva-hyb.jca99} and subsequently
studied in \cite{MonSva10-1,mon.sva-hpe.siam10,MonSva10-2,sol.sva-ine.mor00,Sol-Sv:hy.unif}.
%
%
%The first method of this type was proposed in \cite{sol.sva-hyb.jca99}, and subsequently
%studied, e.g., in
%~\cite{MonSva10-1,mon.sva-hpe.siam10,MonSva10-2,sol.sva-hyb.svva99,sol.sva-hyb.jca99,sol.sva-ine.mor00,Sol-Sv:hy.unif}.
%
%
The key idea consists of observing that \eqref{eq:exact.prox} can be decoupled as
\begin{align}
  \label{eq:dec.prox}
  v_k \in T(z_k),\quad  \lambda_k v_k+z_{k}-z_{k-1}=0,
\end{align}
and then relaxing \eqref{eq:dec.prox}
within relative error tolerance criteria.
Among these new methods, the HPE
method \cite{sol.sva-hyb.svva99} has been shown to be
very effective as a framework for the design and analysis
of many concrete algorithms
(see, e.g.,~\cite{bot.cse-hyb.nfao15,cen.mor.yao-hyb.jota10,eck.sil-pra.mp13,he.mon-acc.siam16,ius.sos-pro.opt10,lol.par.sol-cla.jca09,mon.ort.sva-fir.mpc14,mon.ort.sva-imp.coap14,mon.ort.sva-ada.coap16,MonSva10-2,sol.sva-hyb.svva99,sol.sva-ine.mor00,Sol-Sv:hy.unif}). It can be described as follows: for all $k\geq 1$,
%Let $z_0\in \HH$ and $\sigma\in [0,1[$ be given and iterate for $k\geq 1$,

\begin{align}
%\label{eq:96}
\left\{
       \begin{array}{ll}
			 \label{eq:v.hpe}
			  %\vspace{0.3cm}
				 v_k\in T^{\varepsilon_k}(\tilde z_k),\quad  \norm{\lambda_k v_k+\tilde z_k-z_{k-1}}^2+
 2\lambda_k\varepsilon_k \leq \sigma^2 \norm{\tilde z_k-z_{k-1}}^2,\\[3mm]
%\label{eq:v.hpe}
z_k:=z_{k-1}-\lambda_k v_k,\\
        \end{array}
        \right.
\end{align}
where $\sigma \in [0,1[$. (see Subsection \ref{subsec:moen} for the general notation on $\varepsilon$-enlargements $T^\varepsilon(\cdot)$.)
Note that if $\sigma=0$, then it follows that \eqref{eq:v.hpe} reduces to the exact PP method
\eqref{eq:exact.prox}.
As we mentioned before, recently Bot and Csetnek \cite{bot.cse-hyb.nfao15} proposed and studied
an inertial proximal-like algorithm which combines ideas from \eqref{eq:iner.intr} and
\eqref{eq:v.hpe}. They have proved asymptotic convergence of their method
under the assumption $\alpha(5+4\sigma^2)+\sigma^2<1$ on $\alpha$ and $\sigma$,
where $\sigma\in [0,1[$ is as in \eqref{eq:v.hpe} and $0\leq \alpha_{k-1}\leq \alpha_k\leq \alpha<1$ for all $k\geq 1$ (cf. \eqref{eq:alpha.ser}). This condition enforces $\alpha\approx 0$ whenever $\sigma\approx 1$. This would, in particular, degenerate
the desired inertial effect in many important applications of HPE-type methods for which $\sigma=0.99$ is known (experimentally) to be
the best choice among all possible $\sigma\in [0,1[$ (see, e.g., \cite{eck.sil-pra.mp13,eck.yao-rel.mp17,mon.ort.sva-fir.mpc14,mon.ort.sva-imp.coap14}).

In this paper, we propose an inertial under-relaxed HPE-type method (Algorithm \ref{inertial.hpe}) with guarantee of asymptotic
convergence and iteration-complexity (both pointwise and ergodic) under the assumption (actually more general than)
\eqref{eq:alpha.ser} on $\{\alpha_k\}$; see Assumption ${\bf (A)}$. The price to pay is to perform, in addition to
inertial, under-relaxed
steps.
On the other hand, the under-relaxed parameter $\tau\in ]0,1]$ is explicitly computed
and,  in the case of \eqref{eq:alpha.ser},  $\tau\geq 0.5$, the latter lower bound being uniform on $\sigma\in [0,1[$
(see the third remark following Assumption ${\bf (A)}$).
We also emphasize that our algorithm is different of the corresponding one in \cite{bot.cse-hyb.nfao15}, in the sense it is based on a different
mechanism of iteration.

The main convergence results on Algorithm \ref{inertial.hpe} are Theorems \ref{th:wc}, \ref{th:pic} and
\ref{th:erg}. It seems it is the first time in the literature that global (ergodic) $\mathcal{O}(1/k)$ convergence rates
 are obtained 
for inertial-like proximal algorithms (see Theorem \ref{th:erg}).

\subsection{Forward-backward and Tseng's modified forward-backward methods}

With its roots in the projected gradient algorithm for convex optimization, the \emph{forward-backward method} 
(see, e.g., \cite{lio.mer-spl.sjna79,pas-erg.jmaa79}) is one of the most popular numerical algorithms for solving the structured monotone inclusion problem \eqref{eq:mips}, having numerous applications in modern applied 
mathematics (see, e.g., \cite{bau.com-book}). It can be described as follows: for all $k\geq 1$,
 \begin{align}
    \label{eq:fb.intr}
   z_k:=(\lambda_k B+I)^{-1}(z_{k-1}-\lambda_k F(z_{k-1})),
 \end{align}
where $\lambda_k>0$ is a stepsize parameter and  $z_{k-1}$ is the current iterate.
Under the assumption that $F:\HH\to \HH$ is \emph{cocoercive} and $\{\lambda_k\}$ is within a certain range, it follows that the sequence $\{z_k\}$ generated in \eqref{eq:fb.intr} is weakly convergent
to a solution of \eqref{eq:mips} (see, e.g., \cite{bau.com-book}). 
In the seminal paper \cite{tse-mod.sjco00}, Tseng proposed and studied the following modification of
\eqref{eq:fb.intr} -- known as the \emph{Tseng's modified forward-backward method}: for all $k\geq 1$,
\begin{align}
%\label{eq:96}
\left\{
       \begin{array}{ll}
			 \label{eq:tseng.intr}
			  %\vspace{0.3cm}
				 \tilde z_k:=(\lambda_k B+I)^{-1}(z_{k-1}-\lambda_k F(z_{k-1})),\\[3mm]
%\label{eq:v.hpe}
z_k:=\tilde z_k-\lambda_k(F(\tilde z_k)-F(z_{k-1})).\\
        \end{array}
        \right.
\end{align}
We clearly see that in contrast to \eqref{eq:fb.intr}, \eqref{eq:tseng.intr} performs an additional forward step
to define the next iterate $z_k$.
This is crucial to obtain convergence under the (weaker than cocoercivity)
assumption of \emph{Lipschitz continuity} on $F$ (see, e.g., \cite{bau.com-book,tse-mod.sjco00}).
Since both forward-backward and
Tseng's modified forward-backward methods are known to be special instances of the HPE method
\eqref{eq:v.hpe} for solving \eqref{eq:mip} with $T:=F+B$ 
(see, e.g., \cite{mon.sva-hpe.siam10,sol.sva-hyb.svva99,sva-cla.jota14}), 
we have managed to propose and/or study 
inertial under-relaxed versions of \eqref{eq:fb.intr}  and \eqref{eq:tseng.intr} -- namely, Algorithms
\ref{inertial.tseng} and \ref{inertial.fb}, respectively -- as special instances of the proposed inertial under-relaxed 
HPE method (Algorithm \ref{inertial.hpe}). As a by-product of the results obtained for 
Algorithm \ref{inertial.hpe}, we prove their asymptotic convergence 
as well as their global $\mathcal{O}(1/\sqrt{k})$ pointwise and $\mathcal{O}(1/k)$ ergodic convergence
rates/iteration-complexity (see Theorems \ref{th:tseng.main} and
\ref{th:fb.main}).
We discuss some existing inertial/relaxed variants of \eqref{eq:fb.intr} and \eqref{eq:tseng.intr} as well as how they are related
to Algorithms \ref{inertial.tseng} and \ref{inertial.fb} in the remarks following them.
We also emphasize that, since Algorithms
\ref{inertial.tseng} and \ref{inertial.fb} will be analyzed within the framework of Algorithm \ref{inertial.hpe}, they will automatically inherit all the possible benefits from
the proposed policy of choosing the upper bound on the sequence of inertial parameters and the relaxation parameter 
(see Assumption {\bf (A)}, the remarks following it, and the remarks following Algorithms \ref{inertial.tseng} and
\ref{inertial.fb}).

%\subsection{Related works}

\subsection{The main contributions of this work}
  \label{subsec:main.cont}

We summarize the main contributions of this work are as follows:
\begin{itemize}
 \item[(i)] Asymptotic convergence and nonasymptotic global $\mathcal{O}(1/\sqrt{k})$ \emph{pointwise} and $\mathcal{O}(1/k)$ 
\emph{ergodic} convergence rates (iteration-complexity) of an inertial under-relaxed HPE method (Algorithm \ref{inertial.hpe}) for solving \eqref{eq:mip} under more flexible than existing assumptions on the
choice of inertial $\{\alpha_k\}$ and relative-error $\sigma\in [0,1[$ parameters (see Assumption {\bf (A)} and the remarks 
following it). 
We show, in particular, that it is possible to assume the upper bound
$1/3$ on the sequence of inertial
parameters $\{\alpha_k\}$, which became standard in the analysis of inertial-type proximal algorithms,  at the price of performing under-relaxed iterations with explicitly computed parameter $\tau\geq 0.5$, where the latter lower bound is uniform on the
relative-error parameter $\sigma\in [0,1[$. We also emphasize that, up to the authors knowledge, it is the first time in the literature
that an iteration-complexity analysis is performed for
inertial HPE-type methods (see Theorems \ref{th:pic} and \ref{th:erg}) and it seems it is also the first time
that \emph{ergodic} iteration-complexity results are established for inertial proximal-type algorithms.
\item[(ii)] Asymptotic convergence and pointwise and ergodic iteration-complexity of inertial under-relaxed versions of the
Tseng's modified forward-backward method (Algorithm \ref{inertial.tseng}) and forward-backward method (Algorithm \ref{inertial.fb})
for solving \eqref{eq:mips} under the assumption that $F$ is monotone and either Lipschitz continuous or cocoercive. 
Analogously to (i), in this case, the proposed methods also benefit from the more flexible than standard assumptions on the choice of
inertial parameters (see Subsections \ref{subsec:tsg} and \ref{subsec:fb} for a discussion).
% Moreover, the iteration-complexity analysis performed in this paper for Algorithm \ref{inertial.tseng}  seems to be new.

 \end{itemize}

\subsection{General notation and basics on monotone operators and $\varepsilon$--enlargements}
\label{subsec:moen}

Let $\HH$ be
a real Hilbert space with
inner product $\inner{\cdot}{\cdot}$ and induced norm $\|\cdot\|=\sqrt{\inner{\cdot}{\cdot}\textbf{}}$.
The weak limit of a sequence $\{z_k\}$ in $\HH$ (whenever it exists) will be denoted by $w-\lim_{k\to \infty}\,z_k$.
A set-valued map $T:\HH\tos \HH$ is said to be a \emph{monotone operator} if
$\inner{z-z'}{v-v'}\geq 0$ for all $v\in T(z)$ and $v'\in T(z')$. On the other hand, $T:\HH\tos \HH$ is
\emph{maximal monotone} if $T$ is monotone and its graph
$G(T):=\{(z,v)\in \HH\times \HH\,|\,v\in T(z)\}$ is not properly contained in the graph of any other
monotone operator on $\HH$. The inverse of $T:\HH\tos \HH$ is $T^{-1}:\HH\tos \HH$, defined at any
$z\in \HH$ by $v\in T^{-1}(z)$ if and only if $z\in T(v)$. The resolvent of a maximal monotone operator
$T:\HH\tos \HH$ is $(T+I)^{-1}$ and $z=(T+I)^{-1}x$ if and only if $x-z\in T(z)$. The operator
$\gamma T:\HH\tos \HH$, where $\gamma>0$, is defined by $(\gamma T)z:=\gamma T(z):=\{\gamma v\,|\,v\in T(z)\}$.
%
%A maximal monotone operator $T:\HH\tos \HH$ is said to be \emph{$\mu$-strongly monotone} if $\mu>0$
%and $\inner{z-z'}{v-v'}\geq \mu\norm{z-z'}^2$ for all $v\in T(z)$ and $v'\in T(z')$.

For $T:\HH\tos\HH$ maximal monotone and $\varepsilon\geq 0$, the $\varepsilon$-enlargement~\cite{bur.ius.sva-enl.svva97} of 
$T$
is the operator $T^{\varepsilon}:\HH\tos\HH$ defined by
\begin{align}
 \label{eq:def.teps}
 T^{\varepsilon}(z):=\{v\in \HH\;|\;\inner{z-z'}{v-v'}\geq -\varepsilon\;\;\forall (z',v')\in G(T)\}\quad \forall z\in \HH.
\end{align}
Note that $T(z)\subset T^{\varepsilon}(z)$ for all $z\in \HH$.

%\begin{align}
% \label{eq:def.teps}
%  T^\varepsilon(z):=\{v\in \HH\,\,|\,\,\inner{v-v'}{z-z'}\geq -\varepsilon\quad \forall v'\in T(z')\}.
%\end{align}

The following summarizes some useful properties of $T^{\varepsilon}$ (see, e.g., \cite[Lemma 3.1 and Proposition 3.4(b)]{bur.sva-enl.svaa99}).

\begin{proposition}
\label{pr:teps}
Let $T, S:\HH\tos \HH$ be set-valued maps.  Then,
\begin{itemize}
\item[\emph{(a)}] If $\varepsilon \leq \varepsilon'$, then
$T^{\varepsilon}(z)\subseteq T^{\varepsilon'}(z)$ for every $z \in \HH$.
\item[\emph{(b)}] $T^{\varepsilon}(z)+S^{\,\varepsilon'}(z) \subseteq
(T+S)^{\varepsilon+\varepsilon'}(z)$ for every $z \in \HH$ and
$\varepsilon, \varepsilon'\geq 0$.
\item[\emph{(c)}] $T$ is monotone, if and only if $T  \subseteq T^{0}$.
\item[\emph{(d)}] $T$ is maximal monotone, if and only if $T = T^{0}$.
%\item [\emph{(e)}] if $f:\HH\to\overline{\R}$ is proper, convex and closed, then
%  $\partial_\varepsilon f(x)\subseteq (\partial f)^{\varepsilon}(x)$ for
%  any $\varepsilon \geq 0$ and $x\in \HH$.
\item [\emph{(e)}] If $T$ is maximal monotone, $\{(\tilde z_k,v_k,\varepsilon_k)\}$ is
such that $v_k\in T^{\varepsilon_k}(\tilde z_k)$, for all $k\geq 1$, $w-\lim_{k\to \infty}\,\tilde z_k=z$,
$\lim_{k\to \infty}\,v_k=v$ and $\lim_{k\to \infty}\,\varepsilon_k=\varepsilon$, then
$v\in T^{\varepsilon}(z)$.
\end{itemize}
\end{proposition}

\comment{
\begin{proposition}\emph{(see, e.g., \cite[Lemma 3.1 and Proposition 3.4(b)]{bur.sva-enl.svaa99})}
 \label{pr:ase}
 Assume $\{(\tilde z_k,v_k,\varepsilon_k)\}$ is a sequence in  $\HH\times \HH\times \R_{+}$ such that
$v_k\in T^{\varepsilon_k}(\tilde z_k)$ for all $k\geq 1$. If, $w-\lim_{k\to \infty}\,\tilde z_k=z$,
$\lim_{k\to \infty}\,v_k=v$ and $\lim_{k\to \infty}\,\varepsilon_k=\varepsilon$, then
$v\in T^{\varepsilon}(z)$.
\end{proposition}
}

Next we present the transportation formula
for $\varepsilon$-enlargements.

\begin{theorem}\emph{(see, e.g., \cite[Theorem 2.3]{bur.sag.sva-enl.col99})}
 \label{th:tf}
  Suppose $T:\HH\tos \HH$ is maximal monotone and
	let $\tilde z_\ell, v_\ell\in \HH$, $\varepsilon_\ell, \alpha_\ell\in \R_+$,
	for $\ell=1,\dots, k$, be such that
	 \[
	 v_\ell\in T^{\varepsilon_\ell}(\tilde z_\ell),\quad \ell=1,\dots, k,\quad  \sum_{\ell=1}^k\,\alpha_\ell=1,
	\]
	and define
	\[
	 \tilde {z}_k^a:=\sum_{\ell=1}^k\,\alpha_\ell\, \tilde z_\ell\,,\quad 
	  {v}_k^a:=\sum_{\ell=1}^k\,\alpha_\ell\; v_\ell\,,\quad
	  \varepsilon_k^a:=\sum_{\ell=1}^k\,\alpha_\ell \left(\varepsilon_\ell+\inner{\tilde z_\ell-\tilde z_k^a}
	 {v_\ell- v_k^a}\right).
	\]
        Then, $\varepsilon_k^a\geq 0$ and 
        \[
          v_k^a \in T^{\varepsilon_k^a}(\tilde z_k^a).
        \]
	%
	% Then, the following hold:
	% \begin{itemize}
	% \item[\emph{(a)}] $\varepsilon_k^a\geq 0$ and 
	% $v_k^a \in T^{\varepsilon_k^a}(\tilde z_k^a)$.
	%  \item[\emph{(b)}] If, in addition, $T=\partial f$ for some proper, convex and closed function
          %		$f$ and $v_\ell\in \partial_{\varepsilon_\ell} f(\tilde z_{\ell})$ for $\ell=1,\dots, k$,
         %		then $v_k^a\in \partial_{\varepsilon_k^a} f(\tilde z_k^a)$.
         %\end{itemize}
\end{theorem}

\mgap

The following well-known property will be also useful in this paper.
\mgap
For any $w,z\in \HH$ and $p\in \R$, we have
\begin{align}
 \label{eq:str2}
\norm{pw+(1-p)z}^2=p\norm{w}^2+(1-p)\norm{z}^2-p(1-p)\norm{w-z}^2.
\end{align}

\newpage
\section{An inertial under-relaxed hybrid proximal extragradient method}
 \label{sec:alg}

Consider the monotone inclusion problem \eqref{eq:mip}, i.e., the problem
of finding $z\in \HH$ such that
\begin{align}
 \label{eq:mip2}
 0\in T(z)
\end{align}
where $T$ is a maximal monotone operator on $\HH$ for which $T^{-1}(0)\neq \emptyset$.

In this section, we propose and study the asymptotic convergence and nonasymptotic global convergence
rates (iteration-complexity) of an inertial under-relaxed hybrid proximal extragradient  (HPE) method (Algorithm \ref{inertial.hpe})
for solving \eqref{eq:mip2}.
Regarding the iteration-complexity analysis, we consider the following
notion of approximate solution for \eqref{eq:mip2}: given tolerances $\rho,\epsilon>0$, find
$z,v\in \HH$ and  $\varepsilon\geq 0$ such that
\begin{align}
  \label{eq:appsol}
 v\in T^{\varepsilon}(z),\quad \norm{v}\leq \rho,\quad \varepsilon\leq \epsilon.
\end{align}
Note that $\rho=\epsilon=0$ in \eqref{eq:appsol} gives $0\in T(z)$, i.e., in this case $z\in \HH$ is a solution 
of \eqref{eq:mip2} (for a more detailed discussion on \eqref{eq:appsol}, see, e.g., \cite{mon.sva-hpe.siam10}). 

The main results in this section are Theorems \ref{th:wc}, \ref{th:pic} and \ref{th:erg}. We refer the reader to the remarks and comments
following each of the above mentioned theorems for a discussion regarding the contribution of each of them in the light of related results
available  in the current literature.

\mgap
\mgap

\noindent
\fbox{
\addtolength{\linewidth}{-2\fboxsep}%
\addtolength{\linewidth}{-2\fboxrule}%
\begin{minipage}{\linewidth}%[h]{6.6 in}
\begin{algorithm}
\label{inertial.hpe}
{\bf An inertial under-relaxed HPE method for solving \bf{(\ref{eq:mip2})}}
\end{algorithm}
\begin{itemize}
\item[] {\bf Input:}  $z_0=z_{-1}\in \HH$ and $0\leq \alpha, \sigma<1$ and $0<\tau\leq 1$.
\item [{\bf 1:}] {\bf for} $k=1,2,\dots$,  {\bf do}
\item [{\bf 2:}] Choose $\alpha_{k-1}\in [0,\alpha]$ and define
  \begin{align}
      \label{eq:ext.hpe}
     w_{k-1}:=z_{k-1}+\alpha_{k-1}(z_{k-1}-z_{k-2}).
 \end{align}
\item [{\bf 3:}] Find $(\tilde z_k,v_k,\varepsilon_k)\in \HH\times \HH\times \R_+$ and $\lambda_k>0$ such that
\begin{align}
\label{eq:err.hpe}
 v_k\in T^{\varepsilon_k}(\tilde z_k),\quad \norm{\lambda_k v_k+\tilde z_k-w_{k-1}}^2+2\lambda_k\varepsilon_k\leq
\sigma^2\norm{\tilde z_k-w_{k-1}}^2.
\end{align}
\item[{\bf 4:}] Define
 \begin{align}
  \label{eq:err.hpe2}
   z_k:=w_{k-1}-\tau \lambda_k v_k.
 \end{align}
   \end{itemize}
\noindent
% {\bf end}
\end{minipage}
} %from fbox
\mgap
\mgap

\noindent
{\bf Remarks.}
\begin{itemize}
\item[(i)] Algorithm \ref{inertial.hpe} clearly combines the inertial proximal point (PP) and the HPE methods
\eqref{eq:iner.intr} and  \eqref{eq:v.hpe}, respectively.
It reduces to \eqref{eq:iner.intr} when
$\sigma=0$ and $\tau=1$. Indeed, in this case, using \eqref{eq:err.hpe}, \eqref{eq:err.hpe2} and Proposition \ref{pr:teps}(d),
we find $0\in \lambda_k T(z_k)+z_k-[z_{k-1}+\alpha_{k-1}(z_{k-1}-z_{k-2})]$ for all $k\geq 1$
(cf. iteration $(\mathcal{A}_0)$--$(\mathcal{A}_2)$ in \cite{alv.att-iner.svva01}).
\item[(ii)] A similar inertial relaxed relative-error PP algorithm was proposed and analyzed by Alvarez in \cite{alv-wea.siam03}. We emphasize that in contrast to Algorithm \ref{inertial.hpe}, the algorithm proposed by Alvarez is a 
projective-type algorithm (see, e.g., \cite{sol.sva-hyb.jca99}) and it is based on a different mechanism of iteration.

%
%In this case, the parameter $\beta>0$ plays the role of the standard bound $1/3$ on $\alpha>0$, which appears in the convervence %analysis presented in the latter reference (see also the works \cite{}).
%
%An under-relaxed version of the latter inertial PP method of Alvarez and Attouch can be obtained by letting $\sigma=0$ in Algorithm %\ref{inertial.hpe}. We shall study this method in Section (?) and will show, in particular, that the standard bound $\alpha<1/3$ can be %improved to $\alpha<\beta<1$ at the price of performing an under-relaxed step.
%
\item[(iii)] Algorithm \ref{inertial.hpe} generalizes the HPE method of Solodov and Svaiter \cite{mon.sva-hpe.siam10} and (a special instance of) the under-relaxed HPE method of Svaiter \cite{pre-print-benar}. Indeed, the HPE method is obtained by letting $\alpha=0$ and  $\tau=1$, in which case
$w_{k-1}=z_{k-1}$, while the under-relaxed HPE method (with $t_k\equiv \tau$, in the notation of the latter reference) appears whenever $\alpha=0$ in Algorithm \ref{inertial.hpe}.
\item[(iv)] As we mentioned in Subsection \ref{subsec:hpe.intr}, an inertial HPE-type method was recently proposed
and studied by Bot and Csetnek in \cite{bot.cse-hyb.nfao15}. We refer the reader to Subsection \ref{subsec:hpe.intr} for a discussion of the contributions of this paper in the light of the latter reference, regarding the HPE-type methods.
\item[(v)] We emphasize that, in contrast to the analysis presented in this
work -- see Theorems \ref{th:pic} and \ref{th:erg} --, in all cases of inertial-type algorithms which were mentioned
in remarks (i)--(iv) no iteration-complexity analysis has been obtained.
\item[(vi)]  Step  3 of Algorithm \ref{inertial.hpe}  does  not  specify  how  to  compute $\lambda_k>0$ and  the triple
$(\tilde z_k, v_k,\varepsilon_k)$ satisfying \eqref{eq:err.hpe}, their computation  depending  on
the instance of the method under consideration. In this regard,  Proposition \ref{pr:fb.e.hpe} shows, in particular, how the evaluation of a cocoercive
(monotone) point-to-point operator naturally produces such triples.

\end{itemize}

The next three results, especially Proposition \ref{lm:tech}, will be important
for proving the main results on convergence and iteration-complexity of Algorithm \ref{inertial.hpe}.

\begin{proposition}
\label{inq:err2}
Let $\{z_k\}$, $\{\tilde z_k\}$ and $\{w_k\}$ be generated by \emph{Algorithm \ref{inertial.hpe}}
and define, for all $k\geq 1$,
\begin{align}
  \label{eq:def.thetak}
 s_k:=\max\left\{\eta \norm{z_k-w_{k-1}}^2,(1-\sigma^2)\tau\norm{\tilde z_k-w_{k-1}}^2\right\}
\end{align}
where
\begin{align}
 \label{eq:def.etak}
  \eta:=\eta(\sigma,\tau) :=\dfrac{2}{(1+\sigma)\tau }-1>0.
\end{align}
Then, for any  $z^*\in T^{-1}(0)$,
\begin{align}
 \label{eq:103}
 \norm{z_k-z^*}^2+s_k \leq \norm{w_{k-1}-z^*}^2\qquad \forall k \geq 1.
\end{align}
\end{proposition}
\begin{proof}
 Using \eqref{eq:err.hpe}, \eqref{eq:err.hpe2} and Lemma \ref{pr:ben}(b) we obtain
\begin{align}
  \label{eq:100}
   \norm{z_k-z^*}^2+(1-\sigma^2)\tau \norm{\tilde z_k-w_{k-1}}^2+\tau (1-\tau )\norm{\lambda_k v_k}^2\leq
\norm{w_{k-1}-z^*}^2.
\end{align}
Note now that from  \eqref{eq:err.hpe2} and \eqref{eq:err.hpe} we have
\begin{align*}
  %\label{eq:150}
 \nonumber
 \tau^{-1}\norm{z_k-w_{k-1}}=\norm{\lambda_k v_k}&\leq \norm{\lambda_k v_k+\tilde z_k-w_{k-1}}+
 \norm{\tilde z_k- w_{k-1}}\\
      &\leq (1+\sigma)\norm{\tilde z_k-w_{k-1}},
\end{align*}
which, in turn, gives
\begin{align}
 \label{eq:101}
 (1-\sigma^2)\tau \norm{\tilde z_k-w_{k-1}}^2
%&\geq (1-\sigma^2)\tau_k\dfrac{1}{\tau_k^2(1+\sigma)^2}\norm{z_k-w_{k-1}}^2\\
           &\geq \dfrac{(1-\sigma)}{\tau (1+\sigma)}\norm{z_k-w_{k-1}}^2.
\end{align}
On the other hand, \eqref{eq:err.hpe2} yields
\begin{align}
 \label{eq:102}
 \tau (1-\tau)\norm{\lambda_k v_k}^2=\tau^{-1}(1-\tau)\norm{\tau\lambda_k v_k}^2
      =\tau^{-1}(1-\tau)\norm{z_k-w_{k-1}}^2.
\end{align}
To finish the proof, note that \eqref{eq:103} is a direct consequence of \eqref{eq:def.thetak},
\eqref{eq:100}--\eqref{eq:102} and \eqref{eq:def.etak}.
\end{proof}

\begin{lemma}
 \label{lm:wz}
 Let $\{z_k\}$, $\{w_k\}$ and $\{\alpha_k\}$ be generated by \emph{Algorithm \ref{inertial.hpe}}
and let $z\in \HH$. Then, for all $k\geq 1$,
\begin{align*}
 %\label{eq:202}
 \norm{w_{k-1}-z}^2=(1+\alpha_{k-1})\norm{z_{k-1}-z}^2-\alpha_{k-1}\norm{z_{k-2}-z}^2
+\alpha_{k-1}(1+\alpha_{k-1})\norm{z_{k-1}-z_{k-2}}^2.
\end{align*}
\end{lemma}
\begin{proof}
 From \eqref{eq:ext.hpe} we have $z_{k-1}-z=(1+\alpha_{k-1})^{-1}(w_{k-1}-z)+
 \alpha_{k-1}(1+\alpha_{k-1})^{-1}(z_{k-2}-z)$ and $w_{k-1}-z_{k-2}=(1+\alpha_{k-1})(z_{k-1}-z_{k-2})$, which
combined with the property \eqref{eq:str2} yield the desired identity.
\end{proof}

\begin{proposition}
 \label{lm:tech}
 Let $\{z_k\}$, $\{w_k\}$ and $\{\alpha_k\}$ be generated by \emph{Algorithm \ref{inertial.hpe}}
and let $\{s_k\}$ be as in \eqref{eq:def.thetak}. Let also $z^*\in T^{-1}(0)$ and define
 \begin{align}
 \label{eq:def.phi}
  (\forall k\geq -1)\quad \varphi_k:=\norm{z_k-z^*}^2\;\;\mbox{and}\;\;\;
 (\forall k\geq 1)\quad
\delta_k:=\alpha_{k-1}(1+\alpha_{k-1})\norm{z_{k-1}-z_{k-2}}^2.
\end{align}
Then, $\varphi_0=\varphi_{-1}$ and
 \begin{align}
  \label{eq:ineq.tech}
 \varphi_k-\varphi_{k-1}+s_k \leq
   \alpha_{k-1}(\varphi_{k-1}-\varphi_{k-2})+\delta_k\qquad \forall k\geq 1,
\end{align}
i.e., the sequences $\{\varphi_k\}$, $\{s_k\}$, $\{\alpha_k\}$ and $\{\delta_k\}$ satisfy the assumptions of \emph{Lemma \ref{lm:alv.att}}.
\end{proposition}
\begin{proof}
Using Lemma \ref{lm:wz} with $z=z^*$ and \eqref{eq:def.phi} we
obtain, for all $k\geq 1$,
\begin{align*}
 %\label{eq:202a}
 \norm{w_{k-1}-z^*}^2=(1+\alpha_{k-1})\varphi_{k-1}-\alpha_{k-1}\varphi_{k-2}+\delta_k,
\end{align*}
which combined with Proposition \ref{inq:err2} and the definition of $\varphi_k$ in
\eqref{eq:def.phi} yields \eqref{eq:ineq.tech}. The identity $\varphi_0=\varphi_{-1}$ follows
from the fact that $z_0=z_{-1}$ and the first definition in \eqref{eq:def.phi}.
\end{proof}

Next we present the first result on the asymptotic convergence of Algorithm \ref{inertial.hpe}.

\begin{theorem}[first result on the weak convergence of Algorithm \ref{inertial.hpe}]
 \label{th:main}
 Let $\{z_k\}$, $\{\lambda_k\}$ and $\{\alpha_k\}$ be generated by \emph{Algorithm \ref{inertial.hpe}}.
% and assume that $\lambda_k\geq \underline{\lambda}>0$ for all $k\geq 1$.
%
If the following holds
\begin{align}
 \label{eq:th:main.01}
 \sum_{k=0}^\infty\,\alpha_{k}\norm{z_{k}-z_{k-1}}^2<+\infty
\end{align}
and, additionally, $\lambda_k\geq \underline{\lambda}>0$, for all $k\geq 1$, then the sequence $\{z_k\}$ converges weakly to a solution of  the monotone inclusion problem \eqref{eq:mip2}.% element in $T^{-1}(0)$.
\end{theorem}
\begin{proof}
Using Proposition \ref{lm:tech}, \eqref{eq:th:main.01}, the fact that $\alpha_k\leq \alpha$ for all $k\geq 1$ and Lemma \ref{lm:alv.att}, one concludes that
(i) $\lim_{k\to \infty}\,\|z_k-z^*\|$ exist for every $z^*\in \Omega:=T^{-1}(0)$, and
$\sum_{k=1}^\infty\,s_k<+\infty$, which gives (ii) $\lim_{k\to \infty}\,s_k=0$, where $\{s_k\}$ is as in \eqref{eq:def.thetak}.
%
%\begin{enumerate}
% \item[(i)] $\lim_{k\to \infty}\,\|z_k-z^*\|$ exist for every $z^*\in \Omega:=T^{-1}(0)$.
% \item[(ii)] $\lim_{k\to \infty}\,\theta_k=0$.
%\end{enumerate}
%
In particular, $\{z_k\}$ is bounded.  Using (ii), \eqref{eq:err.hpe}--\eqref{eq:def.thetak} and the assumption
 $\lambda_k\geq \underline{\lambda}>0$ for all $k\geq 1$, we find
\begin{align}
  \label{eq:ter01}
  \lim_{k\to \infty}\,\|z_k-w_{k-1}\|=
 \lim_{k\to \infty}\,\|\tilde z_k-w_{k-1}\|=\lim_{k\to \infty}\,\|v_k\|=\lim_{k\to \infty}\,\varepsilon_k=0.
\end{align}
Now let $z^\infty\in \HH$ be a weak cluster point of $\{z_k\}$ (recall that it is bounded).
Note that it follows from \eqref{eq:ter01} that $z^\infty$ is also a (weak) cluster point of $\{\tilde z_k\}$ 
and let $\{\tilde z_{k_j}\}$ be such that $w-\lim_{j\to \infty}\,\tilde z_{k_j}=z^{\infty}$.
Using \eqref{eq:ter01} and  the inclusion in \eqref{eq:err.hpe} we obtain
\begin{align}
 (\forall j\geq 1)\;\;v_{k_j}\in T^{\varepsilon_{k_j}}(\tilde z_{k_j}),\;\; \lim_{j\to \infty}\,v_{k_j}=0,\;\;
  \lim_{j\to \infty}\,\varepsilon_{k_j}=0\;\;\mbox{and}\;\;  w-\lim_{j\to \infty}\,\tilde z_{k_j}=z^{\infty},
\end{align}
which, in turn, combined with Proposition \ref{pr:teps}(e) yields $z^\infty\in \Omega=T^{-1}(0)$, and so
the desired result follows from (i) and Lemma \ref{lm:opial}.
\end{proof}

\noindent
{\bf Remark.} Condition \eqref{eq:th:main.01} appeared for the first time in \cite{alv.att-iner.svva01}, and since then it has become a standard assumption in the
asymptotic convergence
analysis of different inertial PP-type algorithms. Next, we present a sufficient condition
on the input parameters $(\alpha,\sigma,\tau)$ in Algorithm \ref{inertial.hpe} to ensure \eqref{eq:th:main.01} holds
(see Theorems \ref{th:wc}, \ref{th:pic} and \ref{th:erg}).

\mgap
\mgap

\noindent
{\bf Assumption (A):}
%The sequences $\{z_k\}$, $\{\alpha_k\}$,  $\{w_k\}$, $\{\tilde z_k\}$, $\{v_k\}$,
%$\{\varepsilon_k\}$ and $\{\lambda_k\}$ are generated by Algorithm \ref{inertial.hpe}
%with input
$(\alpha,\sigma,\tau)\in  [0,1[\times [0,1[\times  ]0,1]$ and $\{\alpha_k\}$ satisfy the following (for some $\beta>0$):
\begin{align}
 \label{eq:alpha}
 0\leq \alpha_{k-1}\leq \alpha_{k}\leq \alpha<\beta<1\qquad \forall k\geq 1
\end{align}
and
\begin{align}
\label{eq:betatau}
\tau=\tau(\sigma,\beta):=\dfrac{2(\beta'-1)^2}{(1+\sigma)\left[2(\beta'-1)^2+3\beta'-1\right]},
\end{align}
where
\begin{align}
  \label{eq:betalinha}
 \beta':=\max\left\{\beta,\dfrac{2(1-\sigma)}{3-\sigma+\sqrt{9+2\sigma-7\sigma^2}}\right\}\in
\left[\dfrac{2(1-\sigma)}{3-\sigma+\sqrt{9+2\sigma-7\sigma^2}},1\right[.
\end{align}

\mgap
\mgap

\noindent
{\bf Remarks.}
\begin{itemize}
 \item[(i)] Conditions \eqref{eq:alpha}--\eqref{eq:betalinha} will be crucial to prove convergence and iteration-complexity
               of the algorithms presented and studied in this paper; see, e.g.,
              Theorems \ref{th:wc}, \ref{th:pic} and  \ref{th:erg}, and Section \ref{sec:tseng}.
\item[(ii)] Note that by letting $\sigma=0$, which  by the first remark following Algorithm \ref{inertial.hpe} means
that it reduces to an under-relaxed version of the (exact) Alvarez--Attouch's inertial PP method, we obtain that \eqref{eq:alpha}--\eqref{eq:betalinha}
are now simply given by: $0\leq \alpha_{k-1}\leq \alpha_k\leq \alpha<\beta<1$, for all $k\geq 1$, and
\begin{align}
 \label{eq:betatau22}
\tau=\tau(\beta):=\dfrac{2(\beta'-1)^2}{2(\beta'-1)^2+3\beta'-1},\qquad \beta':=\max\left\{\beta,1/3\right\}\in
\left[1/3,1\right[.
\end{align}
In particular, in this case, we have $\tau=\tau(0,1/3)=1$ whenever $\beta=1/3$ in \eqref{eq:alpha}, which
corresponds to the standard upper bound on $\{\alpha_k\}$ which has been used in different works in the
current literature (see Subsection \ref{subsec:iner.intr} for a discussion). Hence, even in the setting of \emph{exact} inertial PP methods,  conditions
\eqref{eq:alpha}--\eqref{eq:betalinha} generalize the usual assumption \eqref{eq:alpha.ser}. See Figure 1.
\item[(iii)] As we mentioned earlier, an inertial HPE-type method was proposed and studied by
Bot and Csetnek in \cite{bot.cse-hyb.nfao15}, where asymptotic convergence is proved under the assumption $\alpha(5+4\sigma^2)+\sigma^2<1$
on $\alpha,\sigma\in [0,1[$. Note that, in this case, $\alpha\approx 0$ whenever $\sigma\approx 1$. This contrasts to the conditions \eqref{eq:alpha}--\eqref{eq:betalinha}, which, in particular yield
$\tau=\tau(\sigma, 1/3)=1/(1+\sigma)>0.5$ (uniformly on $\sigma$) when $\beta=1/3$ in \eqref{eq:alpha}. 
This may become especially useful in numerical
implementations of Algorithm \ref{inertial.hpe}, since $\sigma=0.99$ has been usually employed in the recent literature
on HPE-type methods (see, e.g., \cite{eck.sil-pra.mp13,eck.yao-rel.mp17,mon.ort.sva-fir.mpc14,mon.ort.sva-imp.coap14}).
Further, \eqref{eq:alpha}--\eqref{eq:betalinha} allow the upper bound $\alpha$  on $\{\alpha_k\}$ to be chosen arbitrarily close to 1, at the price of performing under-relaxed steps with
the explicitly computed $\tau=\tau(\sigma,\beta)$ as in \eqref{eq:betatau}. 
See Figure 1.
%In Figure XXX, we can see the graph of the map $\tau=\tau(\sigma,\cdot)$ for
%$\sigma=0$, $\sigma=0.5$, $\sigma=0.75$ and $\sigma=0.99$.}
%
%\item [(iv)] We mention that $\tau=\tau(\sigma,\beta)$ is given explicitely in \eqref{eq:betatau}, which may be
%important for numerical porpuses.
\end{itemize}

\begin{figure}[H]
\centering
\includegraphics[scale=0.7]{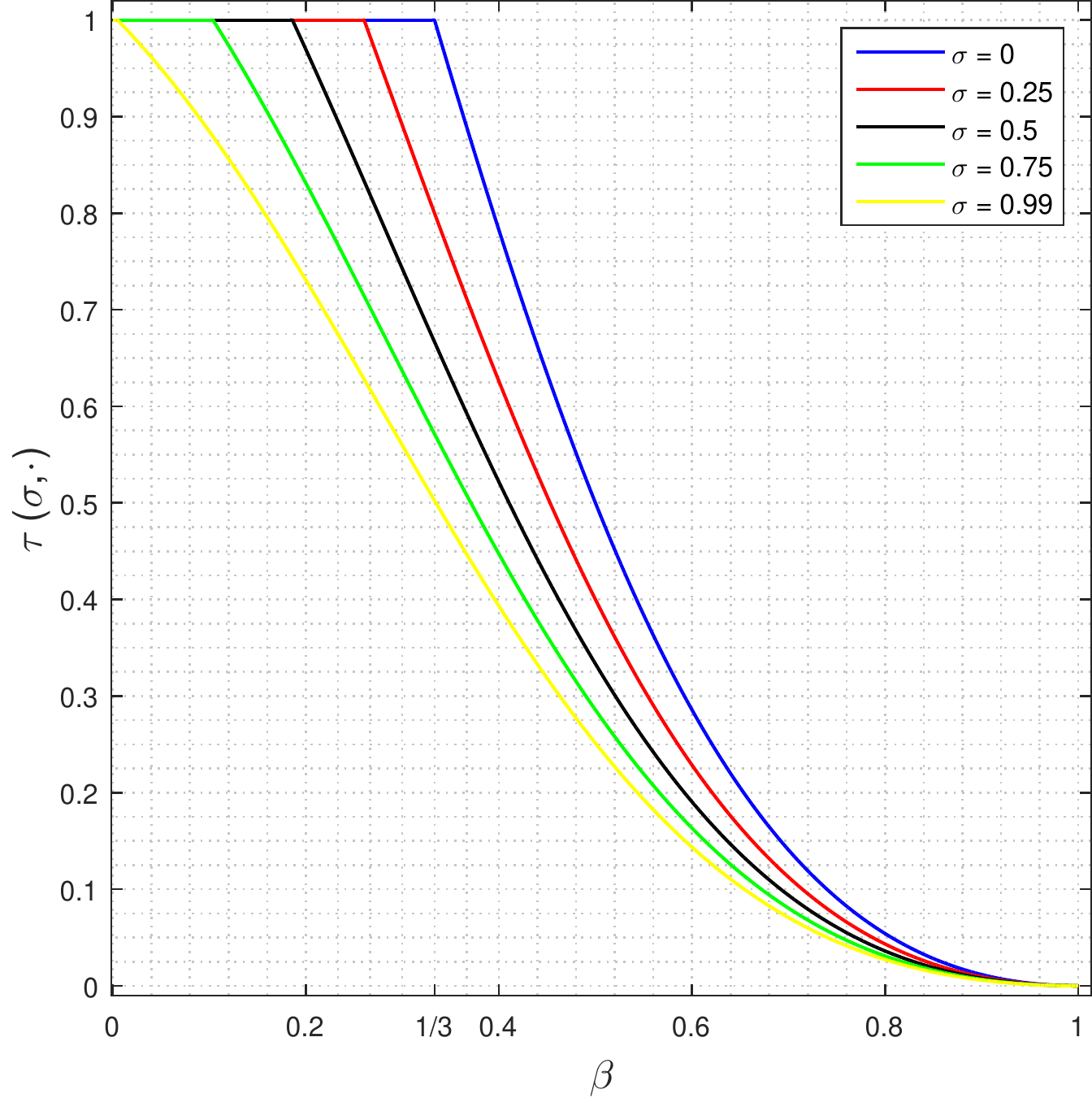}
\caption{Function $]0,1[\ni \beta\mapsto \tau(\sigma, \beta)\in ]0,1[$ as in \eqref{eq:betatau} for $\sigma\in \{0,0.25,0.5,0.75.0.99\}$. Note that $\tau(\sigma, 1/3)\geq 0.5$ for all $\sigma\in [0,1[$. See the second and third remarks following Assumption {\bf (A)}.}
\label{fig:Figure}
\end{figure}

%\newpage

%\mgap
%\mgap

\begin{theorem}[second result on the weak convergence of Algorithm \ref{inertial.hpe}]
 \label{th:wc}
 %
% Let $\{z_k\}$, $\{\alpha_k\}$ and $\{\lambda_k\}$ be generated by \emph{Algorithm \ref{inertial.hpe}}
%with input $(\alpha,\sigma,\tau)$
Under the \emph{Assumption ${\bf (A)}$} on \emph{Algorithm \ref{inertial.hpe}}, let $\eta>0$ be as in \eqref{eq:def.etak} and define the quadratic real function:
\begin{align}
  \label{eq:def.q}
 q(\alpha'):=(\eta-1) \alpha'^{\,\,2}-(1+2\eta)\alpha' +\eta \quad \forall \alpha'\in \R.
\end{align}
Then, $q(\alpha)>0$ and, for every $z^*\in T^{-1}(0)$,
\begin{align}
  \label{eq:sum}
 \sum_{j=1}^k\norm{z_j-z_{j-1}}^2\leq
\dfrac{2\,\norm{z_0-z^*}^2}{(1-\alpha)q(\alpha)} \qquad \forall k\geq 1.
\end{align}
As a consequence, it follows that under the assumption ${\bf (A)}$ the sequence $\{z_k\}$
generated by \emph{Algorithm \ref{inertial.hpe}} converges weakly to a solution of the monotone inclusion
problem \eqref{eq:mip2} whenever $\lambda_k\geq \underline{\lambda}>0$ for all $k\geq 1$.
%
%
%Let $\{z_k\}$, $\{\alpha_k\}$ and $\{\lambda_k\}$ be generated by \emph{Algorithm \ref{inertial.hpe}}
%with input $(\alpha,\sigma,\tau)\in  [0,1[\times [0,1[\times  ]0,1]$ satisfying the following \emph{(}for some $\beta>0$\emph{)}:
%
%\begin{align}
 %\label{eq:alpha}
 %0\leq \alpha_{k-1}\leq \alpha_{k}\leq \alpha<\beta<1\qquad \forall k\geq 1
%\end{align}
%
%and
%
%\begin{align}
%\label{eq:betatau}
%\tau=\tau(\sigma,\beta):=\dfrac{2(\beta'-1)^2}{(1+\sigma)\left[2(\beta'-1)^2+3\beta'-1\right]},
%\end{align}
%
%where
%
%\begin{align}
  %\label{eq:betalinha}
 %\beta':=\max\left\{\beta,\dfrac{2(1-\sigma)}{3-\sigma+\sqrt{9+2\sigma-7\sigma^2}}\right\}\in
%\left[\dfrac{2(1-\sigma)}{3-\sigma+\sqrt{9+2\sigma-7\sigma^2}},1\right[.
%\end{align}
%
%Assume also that $\lambda_k\geq \underline{\lambda}>0$, for all $k\geq 1$.
%
%Then, $q(\alpha)>0$ and, for every $z^*\in T^{-1}(0)$,
%
%\begin{align}
 % \label{eq:sum}
 %\sum_{j=1}^k\norm{z_j-z_{j-1}}^2\leq
%\dfrac{2\,\norm{z_0-z^*}^2}{(1-\alpha)q(\alpha)} \qquad \forall k\geq 1,
%\end{align}
%
%where $q(\cdot)$ is as in \eqref{eq:def.q}.% and $d_0$ denotes the distance of $z_0$ to $T^{-1}(0)$.
\end{theorem}
\begin{proof}
Using \eqref{eq:ext.hpe}, the Cauchy-Schwarz inequality and the Young inequality
$2ab\leq a^2+b^2$ with $a:=\norm{z_k-z_{k-1}}$ and $b:=\norm{z_{k-1}-z_{k-2}}$
we find
\begin{align*}
%\label{eq:300}
 \norm{z_k-w_{k-1}}^2
%&=\norm{z_k-[z_{k-1}+\alpha_k(z_{k-1}-z_{k-2})]}^2\\
   \nonumber
    &=\norm{z_k-z_{k-1}}^2+\alpha_{k-1}^2\norm{z_{k-1}-z_{k-2}}^2-2\alpha_{k-1}\inner{z_{k}-z_{k-1}}{z_{k-1}-z_{k-2}}\\
  \nonumber
    & \geq \norm{z_k-z_{k-1}}^2+\alpha_{k-1}^2\norm{z_{k-1}-z_{k-2}}^2-\alpha_{k-1} \left(2\norm{z_{k}-z_{k-1}}\norm{z_{k-1}-z_{k-2}}\right)\\
    & \geq (1-\alpha_{k-1}) \norm{z_k-z_{k-1}}^2-\alpha_{k-1}(1-\alpha_{k-1})\norm{z_{k-1}-z_{k-2}}^2,
\end{align*}
which combined with \eqref{eq:ineq.tech} and \eqref{eq:def.thetak}, and after some algebraic manipulations, yields
\begin{align}
 \label{eq:301}
 \varphi_k-\varphi_{k-1}-\alpha_{k-1}(\varphi_{k-1}-\varphi_{k-2})-\gamma_{k-1} \norm{z_{k-1}-z_{k-2}}^2
\leq -\eta (1-\alpha_{k-1})\norm{z_k-
  z_{k-1}}^2\quad \forall k\geq 1,
\end{align}
where
\begin{align}
 \label{eq:304}
  \gamma_{k}:=(1-\eta)\alpha_k^2+(1+\eta)\alpha_k\qquad \forall k\geq 0.
\end{align}
 Define,
\begin{align}
 \label{eq:305}
\mu_0:=(1-\alpha_0)\varphi_0\geq 0,\qquad
\mu_k:=\varphi_k-\alpha_{k-1}\varphi_{k-1}+\gamma_{k}\norm{z_k-z_{k-1}}^2\quad \forall k\geq 1,
\end{align}
where $\varphi_k$ is as in \eqref{eq:def.phi}.
 Using \eqref{eq:def.q}, the assumption that $\{\alpha_k\}$ is nondecreasing  (see \eqref{eq:alpha}) and \eqref{eq:301}--\eqref{eq:305} we obtain, for all $k\geq 1$,
\begin{align}
 \label{eq:307}
 \nonumber
 \mu_k-\mu_{k-1}
%&=[\varphi_k-\alpha_{k-1}\varphi_{k-1}+\gamma_{k}\norm{z_k-z_{k-1}}^2]-
%[\varphi_{k-1}-\alpha_{k-2}\varphi_{k-2}+\gamma_{k-1}\norm{z_{k-1}-z_{k-2}}^2]\\
\nonumber
  &\leq
\left[\varphi_k-\varphi_{k-1}-\alpha_{k-1}(\varphi_{k-1}-\varphi_{k-2})-
\gamma_{k-1}\norm{z_{k-1}-z_{k-2}}^2\right]+\gamma_{k}\norm{z_k-z_{k-1}}^2\\
 \nonumber
 &\leq \left[\gamma_{k}-\eta (1-\alpha_{k})\right]\norm{z_k-z_{k-1}}^2\\
\nonumber
&=-\left[(\eta-1)\alpha_{k}^2-(1+2\eta)\alpha_{k}+\eta\right]\norm{z_k-z_{k-1}}^2\\
&=-q(\alpha_k)\norm{z_k-z_{k-1}}^2.
\end{align}
Note now that from \eqref{eq:betatau} and Lemma \ref{lm:inverse} we have
\begin{align*}
 %\label{eq:sigmatau}
 \beta'=\dfrac{4-2(1+\sigma)\tau}{4-(1+\sigma)\tau+\sqrt{(1+\sigma)\tau\left[16-7(1+\sigma)\tau\right]}},
\end{align*}
which, in turn, combined with the definition of $\eta>0$ in \eqref{eq:def.etak}, and after some algebraic calculations, gives
\begin{align*}
 \beta'=\dfrac{2\eta}{2\eta+1+\sqrt{8\eta+1}}.
\end{align*}
The latter identity implies, in particular, that $\beta'$ is either the smallest or the largest root of
the quadratic function $q(\cdot)$.
%
%\begin{align}
%  \label{eq:def.q}
% q(\alpha'):=(\eta-1) \alpha'^{\,\,2}-(1+2\eta)\alpha' +\eta \quad \forall \alpha'\in (0,1).
%\end{align}
%
Hence, from \eqref{eq:alpha} and the fact that $\beta'\geq \beta$ (see \eqref{eq:betalinha})
we obtain
\begin{align*}
 q(\alpha_{k})\geq q(\alpha)>q(\beta')=0.
\end{align*}
The above inequalities combined with \eqref{eq:307} yield
\begin{align}
  \label{eq:ineq.mu}
 \norm{z_k-z_{k-1}}^2\leq \dfrac{1}{q(\alpha)}(\mu_{k-1}-\mu_{k}),\quad \forall k\geq 1,
\end{align}
which, in turn, combined with \eqref{eq:alpha} and the definition of $\mu_k$ in \eqref{eq:305},  gives
\begin{align}
  \label{eq:sum.q}
 \nonumber
 \sum_{j=1}^k\,\norm{z_j-z_{j-1}}^2&\leq \dfrac{1}{q(\alpha)}(\mu_0-\mu_k),\\
             &\leq \dfrac{1}{q(\alpha)}(\mu_0+\alpha \varphi_{k-1}) \quad \forall k\geq 1.
\end{align}
Note now that \eqref{eq:ineq.mu}, \eqref{eq:alpha} and \eqref{eq:305} also yield
\begin{align*}
  \mu_0\geq \ldots \geq \mu_{k}=&\varphi_{k}-\alpha_{k-1}\varphi_{k-1}+\gamma_{k}\|z_{k}-z_{k-1}\|^{2} \\
  %\geq & \varphi_{k}-\alpha_{k} \varphi_{k-1}\\
  \geq&  \varphi_{k}-\alpha\varphi_{k-1},\quad \forall k\geq 1,
\end{align*}
and so,
\begin{equation}
 \label{H13}
    \varphi_{k}\leq \alpha^{k}\varphi_{0}+\frac{\mu_{0}}{1-\alpha}\leq \varphi_{0}+\frac{\mu_{0}}{1-\alpha} \qquad \forall k\geq 0.
\end{equation}
%
%Note that, since $z_0=z_{-1}$ in Algorithm \ref{inertial.hpe}, we have  $\mu_0=(1-\alpha_0)\varphi_0\geq 0$ (see \eqref{eq:def.phi}
%and \eqref{eq:305}).
%
Hence, \eqref{eq:sum} follows directly from \eqref{eq:sum.q}, \eqref{H13}, the definition of $\mu_0$ in \eqref{eq:305} and
the definition of $\varphi_0$ in \eqref{eq:def.phi}.
On the other hand, the second statement of the theorem follows from \eqref{eq:sum} and Theorem \ref{th:main} (recall that
$\alpha_k\leq \alpha<1$ for all $k\geq 0$).
\end{proof}

\mgap
\noindent
{\bf Remark.} A quadratic function similar to $q(\cdot)$, as defined in \eqref{eq:def.q}, was also considered by
Alvarez in \cite{alv-wea.siam03}. As we mentioned in the second remark following Algorithm 1, the algorithm studied in the
later reference is different of the corresponding algorithm presented in this work, namely Algorithm \ref{inertial.hpe}.  Moreover, note that if  $\eta=1$, then $q(\alpha')=1-3\alpha'$ (cf. \cite{alv.att-iner.svva01}).

\mgap

\begin{corollary}
 \label{cor:dom}
   Under the \emph{Assumption ${\bf (A)}$} on \emph{Algorithm \ref{inertial.hpe}},  let $\eta>0$ and $q(\cdot)$ be as in
\eqref{eq:def.etak} and \eqref{eq:def.q}, respectively, and let $z^*\in T^{-1}(0)$.
Then, for all $k\geq 1$,
 \[
   \norm{z_k-z^*}^2+\sum_{j=1}^k\,\tau \Big(\max\left\{\eta \tau \norm{\lambda_jv_j}^2,(1-\sigma^2)
  \norm{\tilde z_j-w_{j-1}}^2\right\}\Big)\leq \left(1+\dfrac{2\alpha(1+\alpha)}{(1-\alpha)^2q(\alpha)}
\right)\norm{z_0-z^*}^2.
 \]
\end{corollary}
\begin{proof}
 Using Proposition \ref{lm:tech} and Lemma \ref{lm:alv.att}(a) we conclude that
 \eqref{eq:alv.att01} holds with $\{s_k\}$, $\{\varphi_k\}$ and $\{\delta_k\}$ as
 in \eqref{eq:def.thetak} and  \eqref{eq:def.phi}, which gives that the desired result follows
from \eqref{eq:alv.att01} and \eqref{eq:sum}.
\end{proof}

\mgap

Next we present the first result on nonasymptotic global convergence rates/iteration-complexity of Algorithm \ref{inertial.hpe}.

\begin{theorem}[global $\mathcal{O}(1/\sqrt{k})$ pointwise convergence rate of Algorithm \ref{inertial.hpe}]
 \label{th:pic}
  Under the \emph{Assumption ${\bf (A)}$} on \emph{Algorithm \ref{inertial.hpe}}, let
%$\{\tilde z_k\}$, $\{v_k\}$, $\{\varepsilon_k\}$ be generated by \emph{Algorithm \ref{inertial.hpe}}
 %  \textcolor{red}{with input} $(\alpha,\sigma,\tau)$ satisfying \eqref{eq:alpha} and \eqref{eq:betatau}, let
 $\eta>0$ and $q(\cdot)$ be as in \eqref{eq:def.etak} and \eqref{eq:def.q}, respectively,
and let $d_0$ denote the distance of $z_0$ to $T^{-1}(0)$. Assume that $\lambda_k\geq \underline{\lambda}>0$
for all $k\geq 1$.
Then, for every $k\geq 1$, there exists $i\in \{1,\dots, k\}$ such that
\begin{align}
  \label{eq:th:pic01}
 &v_i\in T^{\varepsilon_i}(\tilde z_i),\\[3mm]
 \label{eq:th:pic02}
 &\norm{v_i}\leq \dfrac{d_0}{\underline{\lambda}\tau\,\sqrt{k}}
\sqrt{\eta^{-1}\left(1+\dfrac{2\alpha(1+\alpha)}{(1-\alpha)^2q(\alpha)}\right)},\\[3mm]
 \label{eq:th:pic03}
&\varepsilon_i\leq \dfrac{\sigma d_0^2}{2(1-\sigma^2)\underline{\lambda}\tau \,k}
\left(1+\dfrac{2\alpha(1+\alpha)}{(1-\alpha)^2q(\alpha)}\right).
\end{align}
\end{theorem}
\begin{proof}
 Let $z^*\in T^{-1}(0)$ be such that $d_0=\norm{z_0-z^*}$. It follows from Corollary \ref{cor:dom} that, for every
$k\geq 1$, there exists $i\in \{1,\dots, k\}$ such that
 \[
  \tau\, k \Big(\max\left\{\eta \tau \norm{\lambda_iv_i}^2,(1-\sigma^2)
  \norm{\tilde z_i-w_{i-1}}^2\right\}\Big)\leq \left(1+\dfrac{2\alpha(1+\alpha)}{(1-\alpha)^2q(\alpha)}
\right)d_0^2,
 \]
which combined with the assumption $\lambda_i\geq \underline{\lambda}>0$ and \eqref{eq:err.hpe}, and after some
simple algebraic manipulations, yields the desired result.
\end{proof}
%

%\mgap
%\noindent
%{\bf Remark.} Theorem \ref{th:pic} ensures, in particular, that
%for given tolerances $\rho,\epsilon>0$, Algorithm \ref{inertial.hpe} finds
%a triple $(z,v,\varepsilon)$ satisfying \eqref{eq:appsol} after performing
%at most
%
%\begin{align}
% XXXXXX
%\end{align}
%
%iterations.

\noindent
{\bf Remarks.}
\begin{itemize}
 \item[(i)] Theorem \ref{th:pic} provides a global $\mathcal{O}(1/\sqrt{k})$ \emph{pontwise} convergence rate
and ensures, in particular, that
for given tolerances $\rho,\epsilon>0$, Algorithm \ref{inertial.hpe} finds
a triple $(z,v,\varepsilon)$ satisfying \eqref{eq:appsol} after performing
at most
\begin{align*}
  \mathcal{O}\left(\max\left\{\left\lceil \dfrac{d_0^2}{\underline{\lambda}^2\rho^2}\right\rceil,
\left\lceil\dfrac{d_0^2}{\underline{\lambda}\epsilon}\right\rceil\right\}\right)
\end{align*}
 iterations.
  \item[(ii)] If $\alpha=0$ and $\tau=1$, in which case Algorithm \ref{inertial.hpe} reduces to the HPE method of Solodov and Svaiter, then it follows that Theorem \ref{th:pic} reduces to \cite[Theorem 4.4(a)]{mon.sva-hpe.siam10}.
\item[(iii)] Analogous global $\mathcal{O}(1/\sqrt{k})$ \emph{pontwise} convergence rates were also obtained in
\cite{che.cha.ma-ine.sjis15,che.ma.yan-gen.sijo15} for inertial-type algorithms for variational inequality and
convex optimization problems.
\end{itemize}

%\mgap

In order to study the \emph{ergodic} iteration-complexity of Algorithm \ref{inertial.hpe}, we need
to define the following.

The \emph{aggregate stepsize sequence}
$\{\Lambda_k\}$ and the \emph{ergodic sequences}
$\{\tilde {z}_k^a\}$, $\{\tilde v_k^a\}$,
$\{\varepsilon_k^a\}$ associated
to $\{\lambda_k\}$ and
$\{\tilde {z}_k\}$,  $\{v_k\}$, and
$\{\varepsilon_k\}$ are, respectively, for $k\geq 1$,
\begin{align}
\label{eq:d.eg}
  \begin{aligned}
    &\Lambda_k:=\sum_{j=1}^k\, \lambda_j\,,\\
    &\tilde {z}_k^{\,a}:= \frac{1}{\;\Lambda_k}\;
   \sum_{j=1}^k\,\lambda_j\, \tilde {z}_j, \quad
   v_k^{\,a}:= \frac{1}{\;\Lambda_k}\;\sum_{j=1}^k\, \lambda_j\, v_j,\\
   &\varepsilon_k^{\,a}:=
    \frac{1}{\;\Lambda_k}\;\sum_{j=1}^k\,\lambda_j (\varepsilon_j
    +\inner{\tilde {z}_j-
     \tilde {z}_k^{\,a}}{v_j-v_k^{\,a}})=
   \frac{1}{\;\Lambda_k}\;\sum_{j=1}^k\,\lambda_j (\varepsilon_j
    +\inner{\tilde {z}_j-
     \tilde {z}_k^{\,a}}{v_j}).
  \end{aligned}
\end{align}

\mgap

Next we study the \emph{ergodic} iteration-complexity of Algorithm \ref{inertial.hpe} under the assumption
that $\alpha_k\equiv \alpha$ in \eqref{eq:ext.hpe}.

%\newpage
\begin{theorem}[global $\mathcal{O}(1/k)$ ergodic convergence rate of Algorithm \ref{inertial.hpe}]
 \label{th:erg}
  Under the \emph{Assumption ${\bf (A)}$} on \emph{Algorithm \ref{inertial.hpe}}
 and, additionally, the assumption that $\alpha_k\equiv \alpha$, let $\{\tilde z_k^a\}$, $\{v_k^a\}$ and $\{\varepsilon_k^a\}$ be as in \eqref{eq:d.eg} and let
  $d_0$ denote the distance of $z_0$ to $T^{-1}(0)$.
Let also $\eta>0$ and $q(\cdot)$ be as in \eqref{eq:def.etak} and \eqref{eq:def.q}, respectively, and assume
that $\lambda_k\geq \underline{\lambda}>0$ for all $k\geq 1$.

Then, for all $k\geq 1$,
  \begin{align}
   \label{eq:seg08}
   & v_k^a\in T^{\varepsilon_k^a}(\tilde z_k^a),\\[2mm]
    \label{eq:seg07}
   &\norm{v_k^a}\leq \dfrac{2(1+\alpha)d_0}{\underline{\lambda}\tau\,k}
         \sqrt{1+\dfrac{2\alpha(1+\alpha)}{(1-\alpha)^2q(\alpha)}},
\\[2mm]
    \label{eq:seg06}
   &\varepsilon_k^a\leq \dfrac{2\sqrt{2}d_0^2}{\underline{\lambda}\tau\,k}
  \left(1+\dfrac{2\alpha(1+\alpha)}{(1-\alpha)^2q(\alpha)}\right)
  \left(1+\dfrac{\sigma}{\sqrt{(1-\sigma^2)\tau}}+\sqrt{4+\dfrac{(1-\tau)^2}{\eta\tau^2}}\right).
 \end{align}
\end{theorem}
\begin{proof}
Let $z^*\in T^{-1}(0)$ be such that $d_0=\norm{z_0-z^*}$.
Using Algorithm \ref{inertial.hpe}'s definition  and Lemma \ref{pr:ben}(a)
%
%\begin{align*}
% \norm{w-z}^2-\norm{z_+-z}^2\geq (1-\sigma)^2\tau \norm{\tilde z-w}^2+
%2\tau\lambda\left(\varepsilon+\inner{v}{\tilde z-z}\right)+\tau(1-\tau)\norm{\lambda v}^2.
%\end{align*}
%
with $z=\tilde z_k^a$ we find, for all $j\geq 1$,
\begin{align}
%  \nonumber
 \norm{w_{j-1}-\tilde z^a_k}^2-\norm{z_j-\tilde z^a_k}^2
 %&\geq
%(1-\sigma)^2\tau \norm{\tilde z_j-w_{j-1}}^2+
%2\tau \lambda_j\left(\varepsilon_j+\inner{v_j}{\tilde z_j-\tilde z^a_k}\right)
%+\tau(1-\tau)\norm{\lambda_j v_j}^2\\
  \label{eq:dom02}
 &\geq 2\tau \lambda_j\left(\varepsilon_j+\inner{\tilde z_j-\tilde z^a_k}{v_j}\right).
\end{align}
On the other hand,  Lemma \ref{lm:wz} yields
\begin{align*}
  \norm{w_{j-1}-\tilde z_k^a}^2=(1+\alpha_{j-1})\norm{z_{j-1}-\tilde z_k^a}^2-
\alpha_{j-1}\norm{z_{j-2}-\tilde z_k^a}^2+\alpha_{j-1}(1+\alpha_{j-1})\norm{z_{j-1}-z_{j-2}}^2,
\end{align*}
which, in turn, combined with \eqref{eq:dom02} gives, for all $j\geq 1$,
\begin{align*}
 % \label{eq:dom03}
   \norm{z_{j}-\tilde z_k^a}^2-\norm{z_{j-1}-\tilde z_k^a}^2
+2\tau \lambda_j\left(\varepsilon_j+\inner{\tilde z_j-\tilde z^a_k}{v_j}\right)
\leq \alpha_{j-1} \left(\norm{z_{j-1}-\tilde z_k^a}^2-\norm{z_{j-2}-\tilde z_k^a}^2\right)+\delta_j,
\end{align*}
where the sequence $\{\delta_j\}$ is as in \eqref{eq:def.phi}.
%
%\begin{align}
% \delta_j:=\alpha_{j-1}(1+\alpha_{j-1})\norm{z_{j-1}-z_{j-2}}^2\qquad \forall j\geq 1.
%\end{align}
%
Summing the latter inequality over all $j=1,\dots, k$
and using \eqref{eq:d.eg} as well as the assumption $\alpha_k\equiv \alpha$, we obtain
\begin{align*}
 \norm{z_{k}-\tilde z_k^a}^2-\norm{z_{0}-\tilde z_k^a}^2+2\tau \Lambda_k\varepsilon_k^a
\leq \alpha\left(\norm{z_{k-1}-\tilde z_k^a}^2-\norm{z_{-1}-\tilde z_k^a}^2\right)
+\sum_{j=1}^k\,\delta_j,
\end{align*}
which combined with the definition of $\{\delta_j\}$ and \eqref{eq:sum} yields 
(recall that $z_0=z_{-1}$)
\begin{align}
   \label{eq:dom07}
  \nonumber
 2\tau \Lambda_k\varepsilon_k^a-\dfrac{2\alpha(1+\alpha)d_0^2}{(1-\alpha)q(\alpha)}
  &\leq (1-\alpha)\left(\norm{z_{0}-\tilde z_k^a}^2-\norm{z_{k}-\tilde z_k^a}^2\right)\\
 \nonumber
  &+
 \alpha \left(\norm{z_{k-1}-\tilde z_k^a}^2-\norm{z_{k}-\tilde z_k^a}^2\right)\\
 % \nonumber
 %&\leq  \max\left\{\norm{z_{0}-\tilde z_k^a}^2-\norm{z_{k}-\tilde z_k^a}^2,
% \norm{z_{k-1}-\tilde z_k^a}^2-\norm{z_{k}-\tilde z_k^a}^2\right\}
%+\dfrac{2\alpha(1+\alpha)d_0^2}{(1-\alpha)^2q(\alpha)}\\
&\leq  2\max\left\{\norm{z_{0}-\tilde z_k^a}\norm{z_{0}-z_k},
 \norm{z_{k-1}-\tilde z_k^a}\norm{z_{k-1}-z_k}\right\},
\end{align}
where we have also used the inequality $\norm{a}^2-\norm{b}^2\leq 2\norm{a}\norm{a-b}$ for all $a,b\in \HH$.
Now, define
\begin{align}
  \label{eq:dom04}
 \hspace{-3cm}(\forall j\geq 1)\quad \hat z_j:=w_{j-1}-\lambda_j v_j\quad \mbox{and}\quad
\hat z_{k}^a:=\dfrac{1}{\Lambda_k}\,\sum_{j=1}^k\,\lambda_j \hat z_j.
\end{align}
From Corollary \ref{cor:dom}, the first definition
in \eqref{eq:dom04}, \eqref{eq:d.eg}, \eqref{eq:err.hpe2} and the convexity of $\norm{\cdot}^2$ we find
\begin{align}
 \label{eq:seg01}
  \norm{z_\ell-z_j} \leq \norm{z_\ell-z^*}+\norm{z_j-z^*}
 \leq 2 d_0 \sqrt{1+\dfrac{2\alpha(1+\alpha)}{(1-\alpha)^2q(\alpha)}} \qquad \forall \ell, j\geq 0,
\end{align}
\begin{align}
 \label{eq:seg02}
  \hspace{-1cm}(1-\tau)^{-2}\sum_{j=1}^k\,\norm{z_j-\hat z_j}^2=\sum_{j=1}^k\,\norm{\lambda_j v_j}^2
 \leq \dfrac{d_0^2}{\eta\tau^2} \left(1+\dfrac{2\alpha(1+\alpha)}{(1-\alpha)^2q(\alpha)}
\right)
 \end{align}
and
\begin{align}
 \label{eq:seg03}
\nonumber
  \hspace{-0.7cm}
\norm{\tilde z_k^a-\hat z_k^a}^2\leq \dfrac{1}{\Lambda_k}\sum_{j=1}^k\,\lambda_j\norm{\tilde z_j-\hat z_j}^2
  % \nonumber
      %  &
  &\leq \sum_{j=1}^k\,\norm{\lambda_j v_j+\tilde z_j-w_{j-1}}^2\\
  \nonumber
 &\leq  \sigma^2\sum_{j=1}^k\,\norm{\tilde z_j-w_{j-1}}^2\\
&\leq  \dfrac{\sigma^2 d_0^2}{(1-\sigma^2)\tau }
\left(1+\dfrac{2\alpha(1+\alpha)}{(1-\alpha)^2q(\alpha)}\right).
\end{align}
From \eqref{eq:seg01}, \eqref{eq:seg02}, the convexity of $\|\cdot\|^2$ 
and the inequality $\norm{a-b}^2\leq 2\left(\norm{a}^2+\norm{b}^2\right)$ (for all $a,b\in \HH$), we find
\begin{align*}
 % \label{eq:dom06}
  %\nonumber
 \hspace{-1cm}\norm{z_\ell-\hat z_k^a}^2&\leq \dfrac{1}{\Lambda_k}\sum_{j=1}^k\,\lambda_j\norm{z_\ell-\hat z_j}^2\\
  %\nonumber
  &\leq 2\left(\dfrac{1}{\Lambda_k}\sum_{j=1}^k\,\lambda_j\norm{z_\ell-z_j}^2+
    \sum_{j=1}^k\,\norm{z_j-\hat z_j}^2\right)\\
 &\leq
 2d_0^2\left(1+\dfrac{2\alpha(1+\alpha)}{(1-\alpha)^2q(\alpha)}\right)
 \left(4+\dfrac{(1-\tau)^2}{\eta\tau^2}\right)
\qquad \forall \ell\geq 0.
 \end{align*}
Using the above inequality and \eqref{eq:seg03} we obtain, for all $\ell\geq 0$,
\begin{align}
 \label{eq:seg04}
\nonumber
 \norm{z_\ell-\tilde z_k^a}&\leq \norm{z_\ell-\hat z_k^a}+\norm{\tilde z_k^a-\hat z_k^a}\\
   &\leq
   \sqrt{2}d_0\sqrt{1+\dfrac{2\alpha(1+\alpha)}{(1-\alpha)^2q(\alpha)}}
  \left(\dfrac{\sigma}{\sqrt{(1-\sigma^2)\tau}}+\sqrt{4+\dfrac{(1-\tau)^2}{\eta\tau^2}}\right).
  \end{align}
Hence, \eqref{eq:dom07}, \eqref{eq:seg01} with $\ell=0,k-1$ and $j=k$, and \eqref{eq:seg04}  with $\ell=0,k-1$  yield
\begin{align*}
 2\tau \Lambda_k \varepsilon_k^a
 &\leq 4\sqrt{2} d_0^2 \left(1+\dfrac{2\alpha(1+\alpha)}{(1-\alpha)^2q(\alpha)}\right)
  \left(\dfrac{\sigma}{\sqrt{(1-\sigma^2)\tau}}+\sqrt{4+\dfrac{(1-\tau)^2}{\eta\tau^2}}\right)\\
 &+\dfrac{2\alpha(1+\alpha)d_0^2}{(1-\alpha)q(\alpha)}\\
 %&=d_0^2\left[
 %4\left(1+\dfrac{2\alpha(1+\alpha)}{(1-\alpha)^2q(\alpha)}\right)
 % \left(\dfrac{\sigma}{\sqrt{(1-\sigma^2)\tau}}+\sqrt{4+\dfrac{(1-\tau)^2}{\eta\tau^2}}\right)
 %+\dfrac{2\alpha(1+\alpha)}{(1-\alpha)^2q(\alpha)}
%\right]\\
&\leq 4\sqrt{2}d_0^2
  \left(1+\dfrac{2\alpha(1+\alpha)}{(1-\alpha)^2q(\alpha)}\right)
  \left(1+\dfrac{\sigma}{\sqrt{(1-\sigma^2)\tau}}+\sqrt{4+\dfrac{(1-\tau)^2}{\eta\tau^2}}\right),
\end{align*}
which, combined with the assumption $\lambda_k\geq \underline{\lambda}>0$ for all $k\geq 1$, clearly finishes the proof of \eqref{eq:seg06}.

Now note that using \eqref{eq:err.hpe2}, \eqref{eq:ext.hpe}
and the assumption $\alpha_k\equiv \alpha$ we find 
 \begin{align*}
  \tau \lambda_j v_j=z_{j-1}-z_j+\alpha(z_{j-1}-z_{j-2})\qquad \forall j\geq 1.
\end{align*}
Summing the above identity over $j=1,\dots, k$ and using \eqref{eq:d.eg} and \eqref{eq:seg01} 
with $\ell=0$ and $j=k-1,k$
we find (recall that $z_0=z_{-1}$)
\begin{align*}
 \tau \Lambda_k \norm{v_k^a}&\leq \norm{z_0-z_k}+\alpha\norm{z_0-z_{k-1}}\\
  &\leq 2(1+\alpha)d_0\sqrt{1+\dfrac{2\alpha(1+\alpha)}{(1-\alpha)^2q(\alpha)}},
\end{align*}
which, combined with the assumption $\lambda_k\geq \underline{\lambda}>0$ for all $k\geq 1$, yields \eqref{eq:seg07}. 
To finish the proof of the theorem, note that \eqref{eq:seg08} is a direct consequence
of the inclusion in \eqref{eq:err.hpe} and Theorem \ref{th:tf}(a).
\end{proof}

\mgap
\noindent
{\bf Remark.} We mention that, up to the authors knowledge, this is the first time in the literature that  
$\mathcal{O}(1/k)$ global convergence rates are established for inertial PP-type algorithms.

%\newpage
\subsection{On the under-relaxed inertial proximal point method}
 \label{sec:inertial.pp}

In this subsection, we analyze the convergence and iteration-complexity of the
under-relaxed inertial proximal point (PP) method (see, e.g., \cite{alv-wea.siam03,att.cab-con.pre18}) with constant
under-relaxation (Algorithm \ref{inertial.ppm}) for solving \eqref{eq:mip2}.
The analysis is performed by viewing Algorithm \ref{inertial.ppm} within
the framework of Algorithm \ref{inertial.hpe}, for which
asymptotic convergence and iteration-complexity were
obtained in Theorems \ref{th:wc}, \ref{th:pic} and \ref{th:erg}.

%As we mentioned earlier, it improves the standard upper bound $0\leq \alpha<\beta=1/3$ on the ``inertial parameter" $\alpha>0$ to
%$0\leq \alpha<\beta<1$  at the price of performing an under-relaxed step with parameter $\tau \in ]0,1]$.

\mgap
%Below is the algorithm.
\mgap

\noindent
\fbox{
\addtolength{\linewidth}{-2\fboxsep}%
\addtolength{\linewidth}{-2\fboxrule}%
\begin{minipage}{\linewidth}%[h]{6.6 in}
\begin{algorithm}
\label{inertial.ppm}
{\bf Under-relaxed inertial proximal point method for solving \bf{(\ref{eq:mip2})}}
\end{algorithm}
\begin{itemize}
\item[] {\bf Input:} $z_0=z_{-1}\in \HH$ and $0\leq \alpha<1$ and $0<\tau\leq 1$.
\item [{\bf 1:}] {\bf for} $k=1,2,\dots$,  {\bf do}
\item [{\bf 2:}] Choose $\alpha_{k-1}\in [0,\alpha]$ and define
  \begin{align}
      \label{eq:ext.ppm}
     w_{k-1}:=z_{k-1}+\alpha_{k-1}(z_{k-1}-z_{k-2}).
 \end{align}
\item [{\bf 3:}] Compute
\begin{align}
\label{eq:err.ppm}
 \tilde z_k=(\lambda_k T+I)^{-1}w_{k-1}.
\end{align}
\item[{\bf 4:}] Define
 \begin{align}
  \label{eq:err.ppm2}
   z_k:=\tau \tilde z_k+(1-\tau)w_{k-1}.
 \end{align}
   \end{itemize}
\noindent
% {\bf end}
\end{minipage}
} %from fbox

\mgap

%\noindent
%{\bf Remarks.}
%
%\begin{itemize}
% \item[(i)] 
%\end{itemize}

%\mgap

\begin{proposition}
  \label{pr:fer}
 \emph{Algorithm \ref{inertial.ppm}} is a special instance of \emph{Algorithm \ref{inertial.hpe}}
     with $\sigma=0$ in the \emph{Input}, in which case $\varepsilon_k=0$  and
 $v_k=(w_{k-1}-\tilde z_k)/\lambda_k\in T(\tilde z_k)$ for all $k\geq 1$.
\end{proposition}
\begin{proof}
 The proof follows from the well-known fact that $\tilde z=(\lambda T+I)^{-1}w$ if and only if
 $v:=(w-\tilde z)/\lambda\in T(\tilde z)$ and Algorithms \ref{inertial.ppm} and \ref{inertial.hpe}'s definitions.
\end{proof}

%\mgap

\begin{theorem}[convergence and iteration-complexity of Algorithm \ref{inertial.ppm}]
 \label{th:tseng.main}
 Under the \emph{Assumption {\bf (A)}} with $\sigma=0$ on \emph{Algorithm \ref{inertial.ppm}},
let  $\{z_k\}$, $\{v_k\}$,  $\{\tilde z_k\}$ and $\{\lambda_k\}$  be generated
by \emph{Algorithm \ref{inertial.ppm}}  and let the ergodic sequences $\{\tilde z_k^a\}$, $\{v_k^a\}$ and $\{\varepsilon_k^a\}$
be as in \eqref{eq:d.eg}. Let also $q(\cdot)$ be as in \eqref{eq:def.q}
and let $d_0$ denote the distance of $z_0$ to $T^{-1}(0)$. Assume
that $\lambda_k\geq \underline{\lambda}>0$ for all $k\geq 1$. Then, the following statements hold:
\begin{itemize}
 \item[\emph{(a)}] The sequence $\{z_k\}$ converges weakly to a solution of the monotone inclusion
problem \eqref{eq:mip2}.
 \item[\emph{(b)}] For all $k\geq 1$, there exists $i\in \{1,\dots, k\}$ such that
 \begin{align}
   \label{eq:tsg1002}
  v_i\in T(\tilde z_i),\qquad \norm{v_i}\leq \dfrac{d_0}{\underline{\lambda}\tau\,\sqrt{k}}
\sqrt{\eta^{-1}\left(1+\dfrac{2\alpha(1+\alpha)}{(1-\alpha)^2q(\alpha)}\right)}.
\end{align}
\item[\emph{(c)}] If, additionally, $\alpha_k\equiv \alpha$, then, for all $k\geq 1$,
  \begin{align}
   \label{eq:seg0882}
   & v_k^a\in T^{\varepsilon_k^a}(\tilde z_k^a),\\[2mm]
    \label{eq:seg077}
   &\norm{v_k^a}\leq \dfrac{2(1+\alpha)d_0}{\underline{\lambda}\tau\,k}
         \sqrt{1+\dfrac{2\alpha(1+\alpha)}{(1-\alpha)^2q(\alpha)}},
\\[2mm]
    \label{eq:seg0662}
   &\varepsilon_k^a\leq \dfrac{2\sqrt{2}d_0^2}{\underline{\lambda}\tau\,k}
  \left(1+\dfrac{2\alpha(1+\alpha)}{(1-\alpha)^2q(\alpha)}\right)
  \left(1+\sqrt{4+\dfrac{(1-\tau)^2}{\tau(2-\tau)}}\right).
 \end{align}
\end{itemize}
\end{theorem}
\begin{proof}
 The results in (a), (b) and (c) follow directly from Proposition \ref{pr:fer} and Theorems \ref{th:wc}, \ref{th:pic}
and \ref{th:erg}.
\end{proof}

\section{Inertial under-relaxed forward-backward and Tseng's modified forward-backward  methods}
 \label{sec:tseng}

Consider the structured monotone inclusion problem \eqref{eq:mips}, i.e., the problem of finding $z\in \HH$ such that
\begin{align}
 \label{eq:mipt}
 0\in F(z)+B(z)=:T(z)
\end{align}
where $F:D(F)\subset \HH\to \HH$ is point-to-point monotone and $B:\HH\tos \HH$ is a (set-valued) maximal monotone
operator for which $T^{-1}(0)\neq \emptyset$ (precise assumption on $F$ and $B$ will be stated later).

In this section, we study the convergence and iteration-complexity of inertial (under-relaxed) versions
of the forward-backward and Tseng's modified forward-backward methods \eqref{eq:fb.intr} and \eqref{eq:tseng.intr}, respectively, for solving \eqref{eq:mipt}, by 
viewing them within the framework of Algorithm \ref{inertial.hpe}, for which asymptotic convergence
and iteration-complexity were studied in Section \ref{sec:alg}.

\subsection{An inertial under-relaxed Tseng's modified forward-backward method}
  \label{subsec:tsg}

In this subsection, we consider the monotone inclusion problem \eqref{eq:mipt} where the following assumptions are assumed to
 hold:
\begin{itemize}
 \item[(C1)]  $F:D(F)\subset \HH\to \HH$ is monotone and $L$-Lipschitz continuous on a (nonempty) closed convex set $\Omega$
such that $D(B)\subset \Omega\subset D(F)$, i.e., $F$ is monotone on $\Omega$ and there exists $L\geq 0$ such that
 \begin{align*}
  %\label{eq:f.Lip}
  \norm{F(z)-F(z')}\leq L\norm{z-z'}\qquad \forall z,z'\in \Omega.
 \end{align*}
 %
 %\item[(E3)] $F_2:\HH\to \HH$ is $\eta-$cocoercivo, i.e., there exists $\eta>0$ such that
 %
 %\begin{align}
 % \label{eq:f.coco}
 % \inner{F_2(z)-F_2(z')}{z-z'}\geq \eta\norm{F_2(z)-F_2(z')}^2\qquad \forall z,z'\in \HH.
% \end{align}
 %
\item[(C2)] $B$ is a (set-valued) maximal monotone operators on $\HH$.
\item[(C3)] The solution set of \eqref{eq:mipt} is nonempty.
\end{itemize}

We mention that it was proved in  \cite[Proposition A.1]{MonSva10-1} that under assumptions $(C1)$--$(C3)$
the operator $T(\cdot)$ defined in \eqref{eq:mipt} is maximal monotone, which guarantee
that \eqref{eq:mipt} is a special instance of \eqref{eq:mip2}. In particular,  it follows that Algorithm \ref{inertial.hpe} can be used
to solving the structured monotone inclusion \eqref{eq:mipt}. 

As we mentioned earlier, in this subsection, we shall study the convergence and
iteration-complexity of the following inertial under-relaxed version of the
Tseng's modified forward-backward method for solving \eqref{eq:mipt}.

\mgap
\mgap

\noindent
\fbox{
\addtolength{\linewidth}{-2\fboxsep}%
\addtolength{\linewidth}{-2\fboxrule}%
\begin{minipage}{\linewidth}%[h]{6.6 in}
\begin{algorithm}
\label{inertial.tseng}
{\bf An inertial under-relaxed Tseng's modified forward-backward method for solving \bf{(\ref{eq:mipt})}}
\end{algorithm}
\begin{itemize}
\item[] {\bf Input:} $z_0=z_{-1}\in \HH$, $0\leq \alpha<1$, $0<\sigma<1$  and $0<\tau\leq 1$.
\item [{\bf 1:}] {\bf for} $k=1,2,\dots$,  {\bf do}
\item [{\bf 2:}] Choose $\alpha_{k-1}\in [0,\alpha]$ and define
  \begin{align*}
      %\label{eq:itse}
     w_{k-1}:=z_{k-1}+\alpha_{k-1}(z_{k-1}-z_{k-2}).
 \end{align*}
\item [{\bf 3:}] Choose $\lambda_k\in ]0,\sigma/L]$, let $w'_{k-1}=P_{\Omega}(w_{k-1})$ and compute
\begin{align*}
 %\label{eq:itse02}
 &\tilde z_k=(\lambda_k B+I)^{-1}(w_{k-1}-\lambda_k F(w'_{k-1})),\\[2mm]
 &\hat z_k=\tilde z_k-\lambda_k\left(F(\tilde z_k)-F(w'_{k-1})\right).
\end{align*}
\item[{\bf 4:}] Define
 \begin{align*}
%  \label{eq:itse03}
   z_k:= (1-\tau)w_{k-1}+\tau \hat z_k.
 \end{align*}
   \end{itemize}
\noindent
% {\bf end}
\end{minipage}
} %from fbox

\mgap
\mgap

\noindent
{\bf Remarks.}
\begin{itemize}
 \item[(i)] Algorithm \ref{inertial.tseng} reduces to the Tseng's modified forward-backward method \cite{tse-mod.sjco00} for solving
\eqref{eq:mipt} if $\alpha=0$ and $\tau=1$, in which case $w_{k-1}=z_{k-1}$ and $z_k=\hat z_k$. 
 \item[(ii)] An inertial Tseng's modified forward-backward-type method (based on a different mechanism of iteration) was proposed and studied in \cite{bot.cse-hyb.nfao15}.
                The proposed Tseng's modified forward-backward type method in the latter reference tends to suffer 
 from similar limitations as the inertial HPE-type method proposed in \cite{bot.cse-hyb.nfao15}, as we discussed in 
the third remark following Assumption ${\bf (A)}$. Moreover, in contrast to this paper which performs the iteration-complexity analysis
of Algorithm \ref{inertial.tseng} (see Theorem \ref{th:tseng.main}), \cite{bot.cse-hyb.nfao15} has focused on asymptotic convergence.
\end{itemize}

Since the proof of the next proposition follows the same outline of \cite[Proposition 6.1]{mon.sva-hpe.siam10}, we omit it here.

\begin{proposition}
 \label{pr:tsg.e.ihpe}
Let $\{w_k\}$, $\{w'_k\}$, $\{z_k\}$, $\{\alpha_k\}$, $\{\tilde z_k\}$ and $\{\lambda_k\}$ be generated by
\emph{Algorithm \ref{inertial.tseng}} and define
\begin{align}
  \label{eq:tsg101}
 \varepsilon_k:=0\;\;\mbox{and}\;\;
  v_k:=F(\tilde z_k)-F(w'_{k-1})+\dfrac{1}{\lambda_k}(w_{k-1}-\tilde z_k) \qquad \forall k\geq 1.
\end{align}
Then, the sequences $\{w_k\}$, $\{z_k\}$, $\{\alpha_k\}$, $\{\tilde z_k\}$, $\{v_k\}$, $\{\varepsilon_k\}$ and
$\{\lambda_k\}$ satisfy the conditions \eqref{eq:ext.hpe}--\eqref{eq:err.hpe2} in \emph{Algorithm \ref{inertial.hpe}}.
As a consequence, it follows that \emph{Algorithm \ref{inertial.tseng}} is a special instance of \emph{Algorithm \ref{inertial.hpe}}
for solving \eqref{eq:mipt}.
 \end{proposition}

Next we present the convergence and iteration-complexity of Algorithm \ref{inertial.tseng} under
the Assumption ${\bf (A)}$ on the Input $(\alpha,\sigma,\tau)\in [0,1[\times ]0,1[\times ]0,1]$ and on the sequence $\{\alpha_k\}$. We also mention that the observations regarding the parameter $\tau$ in the third remark following
Assumption ${\bf (A)}$ obviously apply to Algorithm \ref{inertial.tseng}.

\begin{theorem}[convergence and iteration-complexity of Algorithm \ref{inertial.tseng}]
 \label{th:tseng.main}
Under the \emph{Assumption ${\bf (A)}$} on $(\alpha,\sigma,\tau)\in [0,1[\times ]0,1[\times ]0,1]$ and $\{\alpha_k\}$,
let $\{z_k\}$, $\{\tilde z_k\}$ and $\{\lambda_k\}$  be generated
by \emph{Algorithm \ref{inertial.tseng}}, let $\{v_k\}$
and $\{\varepsilon_k\}$  be as in \eqref{eq:tsg101} and let the ergodic sequences $\{\tilde z_k^a\}$, $\{v_k^a\}$ and $\{\varepsilon_k^a\}$
be as in \eqref{eq:d.eg}. Let also $\eta>0$ and $q(\cdot)$ be as in \eqref{eq:def.etak} and \eqref{eq:def.q}, respectively,
let $d_0$ denote the distance of $z_0$ to $(F+B)^{-1}(0)$  and assume
that $\lambda_k\geq \underline{\lambda}>0$ for all $k\geq 1$. Then, the following statements hold:
\begin{itemize}
 \item[\emph{(a)}] The sequence $\{z_k\}$ converges weakly to a solution of the monotone inclusion
problem \eqref{eq:mipt}.% whenever $\lambda_k\geq \underline{\lambda}>0$ for all $k\geq 1$.
 \item[\emph{(b)}] For all $k\geq 1$, there exists $i\in \{1,\dots, k\}$ such that
 \begin{align}
   \label{eq:tsg100}
  v_i\in (F+B)(\tilde z_i),\qquad \norm{v_i}\leq \dfrac{d_0}{\underline{\lambda}\tau\,\sqrt{k}}
\sqrt{\eta^{-1}\left(1+\dfrac{2\alpha(1+\alpha)}{(1-\alpha)^2q(\alpha)}\right)}.
\end{align}
%\begin{align}
%  \label{eq:th:pic03}
% &v_i\in T^{\varepsilon_i}(\tilde z_i),\\[3mm]
% \label{eq:th:pic04}
% &\norm{v_i}\leq \dfrac{d_0}{\underline{\lambda}\tau\,\sqrt{k}}
%\sqrt{\eta^{-1}\left(1+\dfrac{2\alpha(1+\alpha)}{(1-\alpha)^2q(\alpha)}\right)},\\[3mm]
% \label{eq:th:pic05}
%&\varepsilon_i\leq \dfrac{\sigma d_0^2}{2(1-\sigma^2)\underline{\lambda}\tau \,k}
%\left(1+\dfrac{2\alpha(1+\alpha)}{(1-\alpha)^2q(\alpha)}\right).
%\end{align}
%
\item[\emph{(c)}] If, additionally, $\alpha_k\equiv \alpha$, then, for all $k\geq 1$,
    \begin{align}
  \begin{aligned}
 \label{eq:seg088}
   & v_k^a\in (F+B)^{\varepsilon_k^a}(\tilde z_k^a),\\[2mm]
    %\label{eq:seg077}
   &\norm{v_k^a}\leq \dfrac{2(1+\alpha)d_0}{\underline{\lambda}\tau\,k}
         \sqrt{1+\dfrac{2\alpha(1+\alpha)}{(1-\alpha)^2q(\alpha)}},
\\[2mm]
    %\label{eq:seg066}
   &\varepsilon_k^a\leq \dfrac{2\sqrt{2}d_0^2}{\underline{\lambda}\tau\,k}
  \left(1+\dfrac{2\alpha(1+\alpha)}{(1-\alpha)^2q(\alpha)}\right)
  \left(1+\dfrac{\sigma}{\sqrt{(1-\sigma^2)\tau}}+\sqrt{4+\dfrac{(1-\tau)^2}{\eta\tau^2}}\right).
 \end{aligned}
 \end{align}
  \end{itemize}
\end{theorem}
\begin{proof}
 The proof follows directly from Proposition \ref{pr:tsg.e.ihpe} and Theorems \ref{th:wc}, \ref{th:pic}
and \ref{th:erg}.
\end{proof}

\mgap
\mgap

\noindent
 {\bf Remarks.}
\begin{itemize}
\item[(i)] Itens (b) and (c) ensure, respectively, global pointwise $\mathcal{O}(1/\sqrt{k})$ and ergodic $\mathcal{O}(1/k)$ convergence rates for
Algorithm \ref{inertial.tseng}. On the other hand, note that the inclusion in \eqref{eq:seg088} is potentially weaker than the corresponding
one in \eqref{eq:tsg100}.
\item[(ii)] If $\lambda_k\equiv \sigma/L$ in Step 3 of Algorithm \ref{inertial.tseng}, in which case $\underline{\lambda}=\sigma/L$, then
$d_0/\underline{\lambda}$ in \eqref{eq:tsg100} and \eqref{eq:seg088} can be replaced by
$d_0 L/\sigma$. In this case, Item (b) gives that for a given tolerance $\rho>0$,  Algorithm \ref{inertial.tseng} finds a pair $(z,v)$
such that (cf. \eqref{eq:appsol})
 \[
   v\in (F+B)(z),\quad \norm{v}\leq \rho
 \]
in at most 
\[
 \mathcal{O}\left(\left\lceil\dfrac{d_0^2 L^2}{\rho^2}\right\rceil\right)
\]
iterations, an analogous remark also holding for Item (c).
\end{itemize}

\subsection{On the inertial under-relaxed forward-backward method}
  \label{subsec:fb}
Similarly to Subsection \ref{subsec:tsg}, in this subsection, we consider the monotone inclusion problem
\eqref{eq:mipt} but now assume the following: (C2) and (C3) as in Subsection \ref{subsec:tsg} and instead of
(C1):
\begin{itemize}
 \item[(C1$'$)]  $F: \HH\to \HH$ is $(1/L)-$cocoercive, i.e., there exists $L>0$ such that
 \begin{align}
  \label{eq:f.coco}
  \inner{z-z'}{F(z)-F(z')}\geq \dfrac{1}{L}\norm{F(z)-F(z')}^2\qquad \forall z,z'\in \HH.
  \end{align}
 \end{itemize}
We observe that it follows from \eqref{eq:f.coco} that $F$ is, in particular, $L$--Lipschitz continuous.

\mgap
\mgap

\noindent
\fbox{
\addtolength{\linewidth}{-2\fboxsep}%
\addtolength{\linewidth}{-2\fboxrule}%
\begin{minipage}{\linewidth}%[h]{6.6 in}
\begin{algorithm}
\label{inertial.fb}
{\bf Inertial under-relaxed forward-backward method for solving \bf{(\ref{eq:mipt})}}
\end{algorithm}
\begin{itemize}
\item[] {\bf Input:} $z_0=z_{-1}\in \HH$, $0\leq \alpha<1$, $0<\sigma<1$ and $0<\tau\leq 1$.
\item [{\bf 1:}] {\bf for} $k=1,2,\dots$,  {\bf do}
\item [{\bf 2:}] Choose $\alpha_{k-1}\in [0,\alpha]$ and define
  \begin{align*}
%      \label{eq:ifb}
     w_{k-1}:=z_{k-1}+\alpha_{k-1}(z_{k-1}-z_{k-2}).
 \end{align*}
\item [{\bf 3:}] Choose $\lambda_k\in ]0,2\sigma^2/L]$ and compute
\begin{align*}
%\label{eq:ifb02}
 &\tilde z_k=(\lambda_k B+I)^{-1}(w_{k-1}-\lambda_k F(w_{k-1})).%,\\[2mm]
% &\hat z_k=\tilde z_k-\lambda_k\left(F(\tilde z_k)-F(w_{k-1})\right).
%
\end{align*}
\item[{\bf 4:}] Define
 \begin{align*}
%  \label{eq:ifb03}
   z_k:= (1-\tau)w_{k-1}+\tau \tilde z_k.
 \end{align*}
   \end{itemize}
\noindent
% {\bf end}
\end{minipage}
} %from fbox

\mgap
\mgap

\noindent
{\bf Remarks.}
\begin{itemize}
\item[(i)] If $\alpha=0$ and $\tau=1$, then it follows that Algorithm \ref{inertial.fb} reduces 
to the forward-backward \cite{lio.mer-spl.sjna79, pas-erg.jmaa79} method for solving \eqref{eq:mipt}.
\item[(ii)] Inertial versions of the forward-backward method were previously proposed and studied in \cite{mou.oli-con.jcam03}, \cite{lor.poc-ine.jmiv15} and \cite{att.cab-con.pre218}. Asymptotic convergence of the forward-backward method proposed in \cite{lor.poc-ine.jmiv15} was
proved in the latter reference, in particular, under the assumption: $0\leq \alpha_{k-1}\leq \alpha_k\leq \alpha<1$, for all $k\geq 1$, and
\[
  \alpha=\alpha(\gamma):=1+\dfrac{\sqrt{9-4\gamma-2\varepsilon\gamma}-3}{\gamma},
\]
for some $\varepsilon\in ]0,(9-4\gamma)/(2\gamma)[$, where $\gamma\in (0,2)$ and $\lambda_k\equiv \lambda:=\gamma/L$
($\gamma=2\sigma^2$ in the notation of the present paper). The apparent limitation of this approach is that
$\alpha\to 0$ if $\gamma\to 2$, i.e., the inertial effect degenerates for large values of the stepsize (see Fig. 1 in \cite{lor.poc-ine.jmiv15}).
 This contrasts to the approach proposed in this paper, where the under-relaxation parameter $\tau\in [0,1[$ is crucial to allowing $\alpha$ sufficiently close to 1, even for large stepsize values, i.e., when $\sigma\approx 1$ (see Assumption {\bf (A)} and part of the
discussion in the third remark following it).
\item[(iii)] Algorithm \ref{inertial.fb} is a special instance (with constant relaxation) of the RIFB algorithm in
\cite{att.cab-con.pre218}. We refer the reader to \cite{att.cab-con.pre218} (see, e.g., Theorems 3.8 and 3.15, and
 Remark 3.13) for a comprehensive discussion of the interplay and benefits of inertia and relaxation.
\end{itemize}

Next proposition shows that Algorithm \ref{inertial.fb} is also a special instance of Algorithm \ref{inertial.hpe} for solving
\eqref{eq:mipt}. Since the proof follows the same outline of \cite[Proposition 5.3]{sva-cla.jota14}, we omit it here too.

\begin{proposition}
 \label{pr:fb.e.hpe}
Let $\{\tilde z_k\}$, $\{z_k\}$, $\{w_k\}$ and $\{\lambda_k\}$ be generated by \emph{Algorithm \ref{inertial.fb}}, let
$T=F+B$ be as in \eqref{eq:mipt} and define, for all $k\geq 1$,
\begin{align}
  \label{pr:0606}
  \varepsilon_k:=\dfrac{\norm{\tilde z_k-w_{k-1}}^2}{4L^{-1}}\;\;\mbox{and}\;\; v_k:=\dfrac{w_{k-1}-\tilde z_k}{\lambda_k}.
\end{align}
Then, the following hold for all $k\geq 1$:
 \begin{align}
   \begin{aligned}
   \label{eq:0607}
   & v_k\in (F^{\varepsilon_k}+B)(\tilde z_k)\subset T^{\varepsilon_k}(\tilde z_k),\\[1mm]
   & \lambda_k v_k+\tilde z_k-w_{k-1}=0,\quad
   2\lambda_k\varepsilon_k\leq \sigma^2\norm{\tilde z_k-w_{k-1}}^2,\quad z_k=w_{k-1}-\tau\lambda_k v_k.
 \end{aligned} 
\end{align}
As a consequence of \eqref{eq:0607} and \emph{Algorithm \ref{inertial.fb}}'s definition, it follows that 
\emph{Algorithm \ref{inertial.fb}} is a special instance of \emph{Algorithm \ref{inertial.hpe}} for solving
\eqref{eq:mipt}.
\end{proposition}

\mgap
%\mgap

We finish this section by presenting the convergence and iteration-complexity of Algorithm \ref{inertial.fb}, which are a direct 
consequence of Proposition \ref{pr:fb.e.hpe} and Theorems \ref{th:wc}, \ref{th:pic} and \ref{th:erg}.
We mention that analogous remarks to those made in the Remarks following Theorem \ref{th:tseng.main} also apply here.

\mgap

\begin{theorem}[convergence and iteration-complexity of Algorithm \ref{inertial.fb}]
\label{th:fb.main}
Under the \emph{Assumption ${\bf (A)}$} on $(\alpha,\sigma,\tau)\in [0,1[\times ]0,1[\times ]0,1]$ and $\{\alpha_k\}$,
let $\{z_k\}$, $\{\tilde z_k\}$ and $\{\lambda_k\}$  be generated
by \emph{Algorithm \ref{inertial.fb}}, let $\{v_k\}$
and $\{\varepsilon_k\}$  be as in \eqref{pr:0606} and let the ergodic sequences $\{\tilde z_k^a\}$, $\{v_k^a\}$ and $\{\varepsilon_k^a\}$
be as in \eqref{eq:d.eg}. Let also $\eta>0$ and $q(\cdot)$ be as in \eqref{eq:def.etak} and \eqref{eq:def.q}, respectively,
let $d_0$ denote the distance of $z_0$ to $(F+B)^{-1}(0)$  and assume
that $\lambda_k\geq \underline{\lambda}>0$ for all $k\geq 1$. Then, the following statements hold:
\begin{itemize}
 \item[\emph{(a)}] The sequence $\{z_k\}$ converges weakly to a solution of the monotone inclusion
problem \eqref{eq:mipt}.
 \item[\emph{(b)}] For all $k\geq 1$, there exists $i\in \{1,\dots, k\}$ such that
 \begin{align}
   \begin{aligned}
   \label{eq:tsg1007}
  & v_i\in (F^{\varepsilon_i}+B)(\tilde z_i),\\[3mm]
  & \norm{v_i}\leq \dfrac{d_0}{\underline{\lambda}\tau\,\sqrt{k}}
\sqrt{\eta^{-1}\left(1+\dfrac{2\alpha(1+\alpha)}{(1-\alpha)^2q(\alpha)}\right)},\\[3mm]
  & %\label{eq:th:pic03}
 \varepsilon_i\leq \dfrac{\sigma d_0^2}{2(1-\sigma^2)\underline{\lambda}\tau \,k}
\left(1+\dfrac{2\alpha(1+\alpha)}{(1-\alpha)^2q(\alpha)}\right).
 \end{aligned}
\end{align}
\item[\emph{(c)}] If, additionally, $\alpha_k\equiv \alpha$, then, for all $k\geq 1$,
    \begin{align}
  \begin{aligned}
 \label{eq:seg0887}
   & v_k^a\in (F+B)^{\varepsilon_k^a}(\tilde z_k^a),\\[2mm]
    %\label{eq:seg077}
   &\norm{v_k^a}\leq \dfrac{2(1+\alpha)d_0}{\underline{\lambda}\tau\,k}
         \sqrt{1+\dfrac{2\alpha(1+\alpha)}{(1-\alpha)^2q(\alpha)}},
\\[2mm]
    %\label{eq:seg066}
   &\varepsilon_k^a\leq \dfrac{2\sqrt{2}d_0^2}{\underline{\lambda}\tau\,k}
  \left(1+\dfrac{2\alpha(1+\alpha)}{(1-\alpha)^2q(\alpha)}\right)
  \left(1+\dfrac{\sigma}{\sqrt{(1-\sigma^2)\tau}}+\sqrt{4+\dfrac{(1-\tau)^2}{\eta\tau^2}}\right).
 \end{aligned}
 \end{align}
  \end{itemize}
\end{theorem}
%
%\begin{proof}
% The proof follows directly from Proposition \ref{pr:tsg.e.ihpe} and Theorems \ref{th:wc}, \ref{th:pic}
%and \ref{th:erg}.
%\end{proof}

%\newpage
\section{Concluding remarks}
 \label{sec:cr}

In this paper, we proposed and studied the asymptotic convergence and iteration-complexity of an inertial under-relaxed HPE-type method.
As applications, we proposed and/or studied inertial (under-relaxed) versions of the Tseng's modified forward-backward and forward-backward methods for solving structured monotone inclusion problems with either  Lipschitz continuous or cocoercive operators. 
All the proposed and/or studied algorithms, namely Algorithms \ref{inertial.hpe}, \ref{inertial.ppm},  \ref{inertial.tseng} and
 \ref{inertial.fb} potentially benefit from a specific policy for choosing the upper bound on the sequence of extrapolation parameters, in which case (under) relaxation plays a central role (see Assumption ${\bf (A)}$ and Theorems \ref{th:pic}, \ref{th:erg},  \ref{th:tseng.main} and \ref{th:fb.main});  see also the recent work \cite{att.cab-con.pre218} of Attouch and Cabot.
We also emphasize that, up to the authors knowledge, this is the first time in the literature that nonasymptotic global convergence rates (iteration-complexity) are provided for inertial HPE-type methods, in particular for the proposed inertial Tseng's modified forward-backward method. 
%\qed

\appendix
\section{Auxiliary results}

\comment{
\begin{lemma}\emph{(\cite[Lemma 2.2]{sva-cla.jota14})}
 \label{lm:coco}
 Let $F:\HH\to \HH$ be $\gamma$--cocoercive, for some $\gamma>0$, and let $w,\widetilde z\in \HH$. Then,
 \[
   F(w)\in F^{\varepsilon}(\widetilde z)\quad\mbox{where}\quad \varepsilon:=\dfrac{\norm{\widetilde z-w}^2}{4\gamma}.
 \]
\end{lemma}
}

Next lemma was proved in \cite[Lemma 2.1]{pre-print-benar}. Here, we present a short and direct proof for
the convenience of the reader.
\begin{lemma}[Svaiter]
%\emph{(\cite[Lemma 2.1]{pre-print-benar})}
 \label{pr:ben}
 Let $\tilde z, v, w\in \HH$ and $\lambda>0$, $\varepsilon\geq 0$ and $\sigma\in [0,1[$ be such that
 %Suppose that $z^*\in T^{-1}(0)$ and
%
\begin{align}
  \label{eq:sva.01}
 v\in T^{\varepsilon}(\tilde z),\quad \norm{\lambda v+\tilde z-w}^2+
2\lambda\varepsilon\leq \sigma^2\norm{\tilde z-w}^2.
\end{align}
Let $\tau\in [0,1]$ and define $z_+:=w-\tau\lambda v$.
Then, the following hold:
\begin{itemize}
 \item[\emph{(a)}] For any $z\in \HH$,
\begin{align*}
 \norm{w-z}^2-\norm{z_+-z}^2\geq (1-\sigma)^2\tau \norm{\tilde z-w}^2+
2\tau\lambda\left(\varepsilon+\inner{\tilde z-z}{v}\right)+\tau(1-\tau)\norm{\lambda v}^2.
\end{align*}
 \item[\emph{(b)}] For any $z^*\in T^{-1}(0)$,
 \begin{align*}
 \norm{w-z^*}^2-\norm{z_+-z^*}^2\geq (1-\sigma^2)\tau\norm{\tilde z-w}^2+\tau(1-\tau)\norm{\lambda v}^2.
\end{align*}
\end{itemize}
\end{lemma}
\begin{proof}
(a)  Using the inequality in \eqref{eq:sva.01} and some algebraic manipulations we find, for any $z\in \HH$,
\begin{align}
  \nonumber
 \norm{w-z}^2-\norm{(w-\lambda v)-z}^2&=\norm{\tilde z-w}^2-\norm{\lambda v+\tilde z-w}^2
+2\lambda \inner{\tilde z-z}{v}\\
  \label{eq:sva.03}
  &\geq (1-\sigma^2)\norm{\tilde z-w}^2+2\lambda\left(\varepsilon+\inner{\tilde z-z}{v}\right).
\end{align}
The fact that $z_+=(1-\tau)w+\tau(w-\lambda v)$ and \eqref{eq:str2} yield
%
%\begin{align}
% z_+=(1-\tau)w+\tau(w-\lambda v)
%\end{align}
%
%
\begin{align*}
 \norm{z_+-z}^2&=(1-\tau)\norm{w-z}^2+\tau\norm{(w-\lambda v)-z}^2-\tau(1-\tau)\norm{\lambda v}^2\\
           &=\norm{w-z}^2-\tau\left(\norm{w-z}^2-\norm{(w-\lambda v)-z}^2\right)-\tau(1-\tau)\norm{\lambda v}^2.
\end{align*}
Multiplying \eqref{eq:sva.03} by $\tau\in [0,1]$ and using the latter identity we obtain the desired inequality in (a).

\mgap
\noindent
(b) This is a direct consequence of Item (a), \eqref{eq:def.teps}, the inclusion in \eqref{eq:sva.01} and the fact that
$0\in T(z^*)$.
\end{proof}

\begin{lemma}
 \label{lm:inverse}
 For any $\sigma\in [0,1[$, the inverse function of the scalar map
\begin{align*}
 A:=]0,1+\sigma]\ni t\mapsto \dfrac{4-2t}{4-t+\sqrt{16t-7t^2}}\in \left[\dfrac{2(1-\sigma)}{3-\sigma+\sqrt{9+2\sigma-7\sigma^2}},1\right[=:B
\end{align*}
is given by
\begin{align*}
B\ni \beta \mapsto \dfrac{2(\beta-1)^2}{2(\beta-1)^2+3\beta-1}\in A.
\end{align*}
\end{lemma}

\begin{lemma}[Opial]
 \label{lm:opial}
Let $\emptyset \neq \Omega\subset \HH$ and $\{z_k\}$ be a sequence in $\HH$ such that
$\lim_{k\to \infty}\,\norm{z_k-z^*}$ exist for every $z^*\in \Omega$. If every (sequential) weak cluster
point of $\{z_k\}$ belongs to $\Omega$, then $\{z_k\}$ converges weakly to a point in $\Omega$.
\end{lemma}

The following lemma was essentially proved by Alvarez and Attouch in  \cite[Theorem 2.1]{alv.att-iner.svva01}.

\begin{lemma}%\emph{(inside the proof of \cite[Theorem 2.1]{alv.att-iner.svva01})}
 \label{lm:alv.att}
Let the sequences $\{\varphi_k\}$, $\{s_k\}$, $\{\alpha_k\}$ and $\{\delta_k\}$ in $[0,+\infty[$
and $\alpha\in \R$ be such that
$\varphi_0=\varphi_{-1}$, $0\leq \alpha_{k-1}\leq \alpha<1$ and
\begin{align}
  \label{eq:alv.att02}
\varphi_{k}-\varphi_{k-1}+s_k\leq \alpha_{k-1}(\varphi_{k-1}-\varphi_{k-2})+\delta_k\qquad \forall k\geq 1.
\end{align}
The following hold:
\begin{enumerate}
  \item [\emph{(a)}] For all $k\geq 1$,
 \begin{align}
       \label{eq:alv.att01}
       \varphi_k+\sum_{j=1}^k\,s_j\leq
      \varphi_0+\dfrac{1}{1-\alpha} \sum_{j=1}^k\,\delta_j.
    \end{align}
  \item [\emph{(b)}] If $\sum^{\infty}_{k=1}\delta_k <+\infty$, then $\lim_{k\to \infty}\,\varphi_{k}$ exist, i.e., the sequence $\{\varphi_k\}$ converges to some element in $[0,\infty[$.
\end{enumerate}
\end{lemma}
\begin{proof}
 It was proved in \cite[Theorem 2.1]{alv.att-iner.svva01} that
 $\mathcal{M}:=(1-\alpha)^{-1}\sum_{j=1}^k\delta_j \geq
\sum_{j=1}^k\,[\varphi_j-\varphi_{j-1}]_+$, where $[\cdot]_{+}=\max \{\cdot,0\}.$
Using this, the  assumptions $\varphi_0=\varphi_{-1}$, $0\leq \alpha_{k-1}\leq \alpha<1$ and \eqref{eq:alv.att02}, and some algebraic manipulations we find
\begin{align*}
 \varphi_k+\sum_{j=1}^k\,s_j &\leq \varphi_0+ \alpha\sum_{j=1}^{k-1}[\varphi_j-\varphi_{j-1}]_{+}+\sum_{j=1}^k\,\delta_j\\
     %&=\alpha\left(\sum_{j=1}^{k-1}[\varphi_j-\varphi_{j-1}]_{+}+[\varphi_0-\varphi_{-1}]_{+}\right)+\sum_{j=1}^k\,\delta_j\\
      &\leq \varphi_0+ \alpha \mathcal{M}+(1-\alpha)\mathcal{M}=\varphi_0+\mathcal{M},
    \end{align*}
which proves (a). To finish the proof of the lemma, note that (b) was proved
inside the proof of \cite[Theorem 2.1]{alv.att-iner.svva01}.
\end{proof}

%\newpage

%\bibliographystyle{plain}
%\bibliography{prox}
%
\def\cprime{$'$}

\end{document}